\documentclass[preprint]{elsarticle}

\usepackage{lineno,hyperref}

\journal{Journal of Computational and Applied Mathematics}

\newcommand{\weak}{\rightharpoonup}
\newcommand{\weakstar}{\stackrel{\ast}{\rightharpoonup}}

\newcommand{\vect}[1]{\boldsymbol{#1}}
\newcommand{\dx}{\mathrm{d}x}

\mathchardef\ordinarycolon\mathcode`\:
\mathcode`\:=\string"8000
\begingroup \catcode`\:=\active
  \gdef:{\mathrel{\mathop\ordinarycolon}}
\endgroup

\usepackage{amsmath,amssymb}
\usepackage{bm}
\usepackage{amsthm}
\usepackage{anyfontsize}
\usepackage{pgfbaselayers}
\usepackage{url}
\usepackage{graphicx}
\usepackage{array}
\usepackage{varwidth}
\definecolor{ao(english)}{rgb}{0.0, 0.5, 0.0}
\definecolor{deepgreen}{rgb}{0.1, 0.5, 0.1}
\usepackage{listings}
\usepackage{enumitem}
\usepackage{caption}

\newtheorem{prop}{Proposition}
\newtheorem{theorem}{Theorem}
\newtheorem{definition}{Definition}
\newtheorem{corollary}{Corollary}
\newtheorem{remark}{Remark}
\newtheorem{lemma}[theorem]{Lemma}
\numberwithin{equation}{section}
\begin{document}

\begin{frontmatter}

\title{Numerical analysis of a topology optimization problem for Stokes flow}

\author[1]{I.~P.~A.~Papadopoulos\corref{cor1}\fnref{fn1}}
\cortext[cor1]{Corresponding author}
\ead{ioannis.papadopoulos@maths.ox.ac.uk}
\fntext[fn1]{I.~P.~is supported by the EPSRC Centre for Doctoral Training in Partial Differential Equations: Analysis and Applications [grant  number EP/L015811/1], the EPSRC grant Spectral element methods for fractional differential equations, with applications in applied analysis and medical imaging [grant number EP/T022132/1] and The MathWorks,  Inc.}

\author[1]{E.~S\"uli}
\ead{endre.suli@maths.ox.ac.uk}

\address[1]{Mathematical Institute, University of Oxford, Oxford, OX2 6GG, UK}

\begin{abstract}
T.~Borrvall and J.~Petersson [Topology optimization of fluids in Stokes flow, International Journal for Numerical Methods in Fluids 41 (1) (2003) 77--107] developed the first model for the topology optimization of fluids in Stokes flow. They proved the existence of minimizers in the infinite-dimensional setting and showed that a suitably chosen finite element method will converge in a weak(-*) sense to an unspecified solution. In this work, we prove novel regularity results and extend their numerical analysis. In particular, given an isolated local minimizer to the  infinite-dimensional   problem, we show that there exists a sequence of finite element solutions, satisfying necessary first-order optimality conditions, that \emph{strongly} converges to it. We also provide the first numerical investigation into convergence rates. 
\end{abstract}


\end{frontmatter}


\section{Introduction}
Topology optimization has become an effective technique in structural and additive manufacturing and has found multiple uses in medicine, architecture, and engineering \cite{Adam2019, Jang2008, Liu2018}. The objective is to find the optimal distribution of a fluid or solid within a given domain that minimizes a problem-specific cost functional \cite{Allaire2012, Bendsoe2004}. In this paper we consider a model for topology optimization for fluids proposed by Borrvall and Petersson \cite{Borrvall2003}. Their seminal work has become the foundation for a number of developments in recent years \cite{Alexandersen2020, Alonso2018, Alonso2020, Olesen2005, Gersborg2005, Evgrafov2014, Evgrafov2006, Deng2018, Guest2006, Aage2008, Kreissl2011, Sa2018}. Their goal was to minimize the power dissipation of a fluid that satisfies both the Stokes equations and a volume constraint restricting the proportion of the domain that the fluid can occupy. In their paper, they derived \emph{generalized Stokes equations}, which incorporate the classical velocity and pressure terms but also introduce a variable, $\rho$, that represents the material distribution of the fluid over the given domain. The presence of fluid is indicated by a value of one in the material distribution whereas absence of fluid is represented by a value of zero. It would be ideal for $\rho : \Omega \to \{0,1\}$, in order to remove any ambiguity in the solutions, however, in general, this is a numerically intractable objective. In the Borrvall--Petersson model $\rho :\Omega \to [0,1]$, but the model is regularized with an inverse permeability term, $\alpha$, which favors solutions where $\rho$ is close to zero or one. From the generalized Stokes equations, Borrvall and Petersson formulated an infinite-dimensional nonconvex optimization problem with inequality, PDE and box constraints.  There have been studies on the numerics of a Ginzburg--Landau regularization of the objective functional that can be shown to recover solutions with sharp transitions in the material distribution \cite{Garcke2015, Garcke2016}. Notably, Garcke et al.~\cite{Garcke2015b} derived a posteriori error estimators designed to resolve the interfaces in the material distribution for the Navier--Stokes extension to the Borrvall--Petersson problem. As far as we are aware, there exist only a couple of results dealing with weak(-*) convergence of discretized solutions, as the mesh size tends to zero, to solutions of the Borrvall--Petersson problem on the whole domain \cite{Borrvall2003, Thore2021}. Moreover, there have been no results concerning strong convergence nor the convergence to \emph{all} the isolated minimizers of the problem. 

In the original paper \cite{Borrvall2003}, it is shown that a minimizing velocity and material distribution to the optimization problem exist \cite[Th.~3.1]{Borrvall2003}; however, the minimizer is not necessarily unique \cite[Sec.~4.5]{Borrvall2003}. It is also shown that there exist finite element solutions that converge to a minimizer of the problem \cite[Th.~3.2]{Borrvall2003}. The proven convergence is weak in the approximation of the velocity and weak-* in the material distribution, with no results for the pressure. In addition, Borrvall and Petersson show that the approximation of material distribution strongly converges to a solution in $L^s(\Omega_b)$, $s \in [1,\infty)$, where $\Omega_b$ is any measurable  subset of $\Omega$ in which the material distribution  that solves the infinite-dimensional problem  is equal to zero or one a.e.~\cite[Sec.~3.3]{Borrvall2003}. Weak-* convergence permits large oscillations in the material distribution, called checkerboarding, which could occur in areas where the material distribution is not zero or one under the current results. However, in practice, checkerboarding is not observed in these regions. Since there can be multiple solutions, the nature of the convergence is ambiguous. In particular, it is not clear if there are sequences of finite element solutions converging to every solution  of the infinite-dimensional problem. 

Our goal is to extend and refine the analysis of Borrvall and Petersson. We show that, given an isolated minimizer to the  infinite-dimensional   problem, there exists a sequence of finite element solutions, satisfying the necessary first-order optimality conditions, that strongly converges. In particular, we strengthen the convergence from weak convergence in $H^1(\Omega)^d$ to strong convergence in $H^1(\Omega)^d$ for the velocity, and from weak-* convergence in $L^\infty(\Omega)$ to strong convergence in $L^s(\Omega)$, $s \in [1,\infty)$ for the material distribution. Moreover, in the case of a homogeneous Dirichlet boundary condition, we show that the material distribution is weakly differentiable inside any compact subset of the support of the velocity; more specifically $\rho \in H^1(U_\theta)$, for any $\theta > 0$, where $U_\theta$ is any measurable subset of $\Omega$ in which $|\vect{u}|^2 \geq \theta>0$ a.e.~in $U_\theta$. This analysis confirms that checkerboarding cannot occur under mild assumptions on the model. Hence, isolated minimizers of the problem can be well approximated using the finite element method. We conclude with a numerical investigation into the convergence of the finite element solutions. 

By first considering the optimization problem, we derive necessary first-order optimality conditions in Section \ref{sec:foc}. By construction, the generalized Stokes equations are satisfied but we also show that the material distribution satisfies a variational inequality. In Section \ref{sec:regularity}, we show that under moderate assumptions, the material distribution is weakly differentiable in the case of a homogeneous Dirichlet boundary condition for the velocity. We tackle the issue of multiple local minima in Section \ref{sec:femapprox}, by considering closed balls around isolated local minimizers. In that section, we also prove that for each isolated minimizer there exists a sequence of finite element solutions to the discretized first-order optimality conditions, which strongly converges to the solution  of the infinite-dimensional problem.  In Section \ref{sec:numerical}, we computationally investigate the convergence of sequences of finite element approximations to the respective solutions  of the infinite-dimensional problem.

\section[Existence \& FOC]{Existence and necessary first-order optimality conditions}
\label{sec:foc}
The topology optimization problem of Borrvall and Petersson \cite{Borrvall2003} is as follows: find the velocity, $\vect{u}$, and the material distribution, $\rho$, that solve the minimization problem
\begin{align}
\min_{(\vect{v},\eta) \in H^1_{\vect{g},\mathrm{div}}(\Omega)^d \times C_\gamma} J(\vect{v},\eta),\label{borrvallmin} \tag{BP}
\end{align}
where,
\begin{align*}
J(\vect{v},\eta) &:=\frac{1}{2} \int_\Omega \left(\alpha(\eta) |\vect{v}|^2 + \nu |\nabla \vect{v}|^2 - 2\vect{f}\cdot \vect{v}\right) \dx,\\
H^1_{\vect{g}}(\Omega)^d &:= \{\vect{v} \in H^1(\Omega)^d : \vect{v}|_{\partial \Omega}= \vect{g} \;\; \text{on} \;\; \partial \Omega\},\\
H^1_{\vect{g},\mathrm{div}}(\Omega)^d &:= \{\vect{v} \in H^1_g(\Omega)^d : \mathrm{div}(\vect{v}) = 0 \;\; \text{a.e.~in} \; \Omega \},\\
C_\gamma &:= \left\{ \eta \in L^\infty(\Omega) : 0 \leq \eta\leq 1 \;\; \text{a.e.}, \;\; \int_\Omega \eta \; \dx \leq \gamma |\Omega| \right\},
\end{align*}
where $\Omega \subset \mathbb{R}^d$ is a Lipschitz domain with dimension  $d \in \{2,3\}$,  $\vect{f} \in L^2(\Omega)^d$, $\nu$ is the (constant) viscosity, and $\gamma \in (0,1)$ is the volume fraction. Here, $\alpha$ is the inverse permeability, modeling the influence of the material distribution on the flow. For values of $\rho$ close to one, $\alpha(\rho)$ is small permitting fluid flow; for small values of $\rho$, $\alpha(\rho)$ is very large, restricting fluid flow. The function $\alpha$ is assumed to have the following properties:
\begin{enumerate}[label=({A}\arabic*)]
\setlength\itemsep{0em}
\item $\alpha: [0,1] \to [\underline{\alpha}, \overline{\alpha}]$ with $0 \leq \underline{\alpha}$ and $ \overline{\alpha} < \infty$; \label{alpha1}
\item $\alpha$ is convex and monotonically decreasing; \label{alpha2}
\item $\alpha(0) = \overline{\alpha}$ and $\alpha(1) = \underline{\alpha}$, \label{alpha3}
 \end{enumerate}
generating an operator also denoted $\alpha: C_\gamma \to L^\infty(\Omega; [\underline{\alpha},\overline{\alpha}])$. Typically, in the literature $\alpha$ takes the form \cite{Borrvall2003}
\begin{align}
\alpha(\rho) = \bar{\alpha}\left( 1 - \frac{\rho(q+1)}{\rho+q}\right), \label{eq:alphachoice}
\end{align}
where $q>0$ is a penalty parameter, so that $\lim_{q \to \infty} \alpha(\rho) = \bar{\alpha}(1-\rho)$. Furthermore,  $|_{\partial \Omega}$ is to be understood in the boundary trace sense \cite[Ch.~5.5]{Evans2010}, $\vect{g} \in H^{1/2}(\partial \Omega)^d$, $\vect{g} = \vect{0} $ on $\Gamma \subset \partial \Omega$, with $\mathcal{H}^{d-1}(\Gamma)>0$, i.e.\ $\Gamma$ has nonzero $(d-1)$-dimensional Hausdorff measure. Hence, the Poincar\'e inequality holds with constant $c_p$ such that $\|\vect{v}\|_{L^2(\Omega)}\leq c_p \|\nabla \vect{v} \|_{L^2(\Omega)}$ for all $\vect{v} \in H^1(\Omega)^d$ with $\vect{v}|_{\Gamma} = \vect{0}$.
\begin{remark}
The integral in (\ref{borrvallmin}) is well defined. {Indeed,} since $\alpha$ is assumed to be convex it is {Borel measurable; also since} $\rho \in C_\gamma$ is Lebesgue measurable, the composition {$\alpha(\rho) : \Omega \to [\underline{\alpha}, \overline{\alpha}]$} is Lebesgue measurable.
\end{remark}

The next theorem is due to Borrvall and Petersson \cite[Th.~3.1]{Borrvall2003}.
\begin{theorem}
\label{th:borrvallexistence}
Suppose that $\Omega \subset \mathbb{R}^d$ is a Lipschitz domain,  $d \in \{2,3\}$,  and $\alpha$ satisfies properties \ref{alpha1}--\ref{alpha3}. Suppose in addition that $\alpha \in C^1([0,1];[\underline{\alpha},\bar{\alpha}])$. Then, there exists a pair $(\vect{u}, \rho) \in H^1_{\vect{g},\mathrm{div}}(\Omega)^d \times C_\gamma$ that locally minimizes $J$, as defined in (\ref{borrvallmin}).
\end{theorem}

\begin{remark}
Although a solution exists, it is not necessarily unique, as observed in the numerical examples in Section \ref{sec:numerical}, since the optimization problem is nonconvex. The nonconvexity is caused by the term $\alpha(\rho)|\vect{u}|^2$ in (\ref{borrvallmin}). A rough argument examines the second partial Fr\'echet derivative of $J$ with respect to $\vect{u}$ and $\rho$ (assuming that it exists), i.e., for suitable variations $\vect{v}, \eta$,
\begin{align}
\langle J''_{\vect{u},\rho}(\vect{u},\rho), \vect{v}, \eta \rangle = \int_\Omega \alpha'(\rho) \eta (\vect{u} \cdot \vect{v}) \, \mathrm{d}x.
\end{align}
Note that \ref{alpha2} implies that $\alpha'(\rho) < 0$ a.e.~and the box constraints enforce $\rho \geq 0$ a.e. Thus, given a pair $(\vect{u}, \rho)$, such that $\vect{u} \geq \vect{0}$ a.e., $\vect{u}$, $\rho$ are nonzero functions, and nonzero variations $(\vect{v}, \eta)$, such that $\vect{v} \geq \vect{0}$ a.e.~and $\eta > 0$ a.e., we see that $\langle J''_{\vect{u},\rho}(\vect{u},\rho), \vect{v}, \eta \rangle < 0$. Moreover, if $2 \langle J''_{\vect{u},\rho}(\vect{u},\rho), \vect{v}, \eta \rangle + \langle J''_{\vect{u},\vect{u}}(\vect{u},\rho), \vect{v}, \vect{v} \rangle +\langle J''_{\rho,\rho}(\vect{u},\rho), \eta, \eta \rangle < 0$, then the optimization problem is not convex.
\end{remark}

\begin{prop}
\label{prop:frechet}
Suppose $\alpha$ also satisfies
\begin{enumerate}[label=({A}\arabic*)]
\setcounter{enumi}{3}
\item $\alpha$ is twice continuously differentiable. \label{alpha4}
\end{enumerate}
Then $J : H^1(\Omega)^d \times L^s(\Omega) \to \mathbb{R}$ is partially Fr\'echet differentiable with respect to $\vect{u}$ and  partially Fr\'echet differentiable with respect to $\rho$, where $1 <s \leq \infty$ in two dimensions and $ 3/2 \leq s \leq \infty$ in three dimensions. Moreover, for all $\vect{v} \in H^1_0(\Omega)^d$ and $\eta \in C_\gamma$ we have that
\begin{align}
\langle J'_{\vect{u}}(\vect{u},\rho), \vect{v} \rangle &= \int_\Omega \alpha(\rho) \vect{u} \cdot \vect{v} + \nu \nabla \vect{u} : \nabla \vect{v} - \vect{f}  \cdot \vect{v} \; \dx,\\
\langle J'_{\rho}(\vect{u},\rho), \eta - \rho \rangle &= \frac{1}{2} \int_\Omega \alpha'(\rho) |\vect{u}|^2 (\eta - \rho) \; \dx,
\end{align}
where $J'_{\vect{u}}(\vect{u},\rho)$ denotes the Fr\'echet derivative of $J$ with respect to $\vect{u}$ and $J'_{\rho}(\vect{u},\rho)$ denotes the Fr\'echet derivative of $J$ with respect to $\rho$. 
\end{prop}
\begin{proof}
Proposition \ref{prop:frechet} follows from the definition of Fr\'echet differentiability, the mean value  inequality, the dominated convergence theorem,  and the Sobolev embedding theorem \cite[Prop.~2.3]{Papadopoulos2021e}.
\end{proof}
\begin{remark}
It can be checked that if $\alpha$ is $(n+1)$-times continuously differentiable then $J$ is $n$-times Fr\'echet differentiable with respect to $\vect{u}$ and $\rho$.
\end{remark}

The following proposition is the main result of this section. We show that if $(\vect{u},\rho)$ is a minimizer  of the optimization problem (\ref{borrvallmin}), then the minimizer also satisfies first-order optimality conditions consisting of two equations and a variational inequality. 
\begin{prop}
\label{prop:pressureexistence}
Suppose that $\Omega \subset \mathbb{R}^d$ is a Lipschitz domain,  with $d\in\{2,3\}$, and $\alpha$ satisfies properties \ref{alpha1}--\ref{alpha4}. Consider a local or global minimizer $(\vect{u}, \rho) \in H^1_{\vect{g},\mathrm{div}}(\Omega)^d \times C_\gamma$ of (\ref{borrvallmin}).  Then, there exists a unique Lagrange multiplier $p \in L^2_0(\Omega)$ such that the following necessary first-order optimality conditions hold:
\begin{alignat}{2}
a_\rho(\vect{u},\vect{v}) + b(\vect{v}, p) &= l_{\vect{f}}(\vect{v}) \;\; && \text{for all} \;\; \vect{v} \in H^1_0(\Omega)^d, \tag{FOC1} \label{FOC1}\\
b(\vect{u},q) &=0 \;\; &&  \text{for all} \;\; q \in L^2_0(\Omega), \tag{FOC2} \label{FOC2}\\
c_{\vect{u}}(\rho,\eta-\rho) &\geq 0 \;\; && \text{for all} \;\; \eta \in C_\gamma,  \tag{FOC3} \label{FOC3}
\end{alignat}
where
\begin{align*}
L^2_0(\Omega) &:= \left\{ q \in L^2(\Omega) : \int_\Omega q \; \dx  = 0 \right \},
\end{align*}
and
\begin{alignat*}{2}
a_\rho(\vect{u},\vect{v})&:= \int_\Omega  \left[ \alpha(\rho) \vect{u} \cdot \vect{v} + \nu \nabla \vect{u} : \nabla \vect{v} \right]\dx, \quad
&&l_{\vect{f}}(\vect{v}):= \int_\Omega \vect{f}\cdot\vect{v} \; \dx,\\
b(\vect{v},q)&:= -\int_\Omega q \; \mathrm{div}(\vect{v}) \; \dx, 
&&c_{\vect{u}}(\rho,\eta) := \frac{1}{2} \int_\Omega \alpha'(\rho)\eta |\vect{u}|^2\; \dx.
\end{alignat*}
\end{prop}
\begin{proof}
We will first show that (\ref{FOC1})--(\ref{FOC2}) are satisfied by generalizing arguments, used for the Stokes system with a homogeneous Dirichlet boundary condition, found in \cite{Ozanski2017}. For ease of notation we define $X_{\vect{g}}:=H^1_{\vect{g}}(\Omega)^d$, $X_0 :=H^1_0(\Omega)^d$, $V_{\vect{g}}:=H^1_{\vect{g}, \mathrm{div}}(\Omega)^d$, $V_0:=H^1_{0,\mathrm{div}}(\Omega)^d$ and $M:=L^2_0(\Omega)$. The respective dual spaces of $X_0$, $V_0$ and $M$ are denoted with $^*$. We also define the associated operators, $A\in\mathcal{L}(X_{\vect{g}},X_0^*)$, $B\in\mathcal{L}(X_{\vect{g}},M)$ and $B_0 \in \mathcal{L}(X_0, M)$ by
\begin{align}
\langle A \vect{u}, \vect{v} \rangle := a_\rho(\vect{u},\vect{v}), \;\; \langle B \vect{w}, q \rangle := b(\vect{w}, q), \;\; \text{and} \;\;  \langle B_0 \vect{v}, q \rangle := b(\vect{v}, q).
\end{align}
We note that $\text{ker}(B_0) = V_0$. From Theorem \ref{th:borrvallexistence}, we know that there exists a pair $(\vect{u}, \rho) \in V_{\vect{g}} \times C_\gamma$ that is a local minimizer for (\ref{borrvallmin}). For any given $\vect{v} \in V_0$, we see that $\vect{u} + t\vect{v} \in V_{\vect{g}}$, $t \in \mathbb{R}$. If $(\vect{u},\rho) \in V_{\vect{g}} \times C_\gamma$ is  a minimizer, then, by definition, there exists an $r > 0$ such that, for any $(\vect{w}, \eta) \in V_{\vect{g}} \times C_\gamma$, $(\vect{w},\eta) \neq (\vect{u},\rho)$ that satisfies
\begin{align}
\| \vect{u} - \vect{w}\|_{H^1(\Omega)} + \| \rho - \eta \|_{L^\infty(\Omega)}  <  r
\label{isolatedminr}
\end{align}
we have that $J(\vect{u}, \rho)  \leq  J(\vect{w}, \eta)$. Hence, for any given $\vect{v} \in V_0$, if $0 < t  <  r/\|\vect{v}\|_{H^1(\Omega)}$, the following inequality holds
\begin{align}
\frac{1}{t}(J(\vect{u}+t\vect{v}, \rho) - J(\vect{u},\rho)) \geq 0.
\end{align}
By Proposition \ref{prop:frechet}, $J$ is Fr\'echet differentiable, and therefore also Gateaux differentiable, with respect to $\vect{u}$. Hence as $t \to 0_+$, we see that
\begin{align}
\langle J_{\vect{u}}'(\vect{u},\rho), \vect{v} \rangle \geq 0  \;\; \text{for all} \; \vect{v} \in V_0.
\end{align}
By considering the same reasoning with $t<0$, we deduce that
\begin{align}
\langle J_{\vect{u}}'(\vect{u},\rho), \vect{v} \rangle = 0  \;\; \text{for all} \; \vect{v} \in V_0.
\end{align}
From Proposition \ref{prop:frechet}, we know that $J_{\vect{u}}'(\vect{u},\rho) = A \vect{u} - \vect{f}$ and hence $A \vect{u}-\vect{f} \in V_0^\circ$ where
\begin{align}
V_0^\circ := (\mathrm{ker}(B_0))^\circ = \{h \in X_0^* : \langle h, \vect{v} \rangle = 0 \; \text{for all} \; \vect{v} \in V_0\}.
\end{align}
We know that the operator $B_0$ satisfies the following equivalent version of the inf-sup condition \cite[Ch.~1, Sec.~4.1, Lem.~4.1]{Girault2012}:
\begin{align}
\text{there exists a}\; \beta>0 \; \text{such that, for all} \; q \in M, \; \|B_0^* q\|_{X_0^*} \geq \beta \|q\|_M,
\end{align}
where $B_0^*$ is the dual operator of $B_0$, defined by $\langle \vect{v}, B_0^* q \rangle = \langle B_0 \vect{v}, q \rangle$.
This implies that $B_0^*$ is injective (and therefore bijective) from $M$ into $\mathrm{Im} B_0^*$. Furthermore, it also implies that $(B_0^*)^{-1}$ is continuous. Consider $\vect{f} \in \mathrm{Im} B_0^*$; then, there exists a $q \in M$ such that $\vect{f} = B_0^* q$ and
\begin{align}
\|(B_0^*)^{-1}\vect{f}\|_M \leq \frac{1}{\beta} \|\vect{f}\|_{X_0^*}.
\end{align}
Therefore, $\mathrm{Im} B_0^*$ is closed.

Since $\mathrm{Im} B_0^*$ is closed, by Banach's closed range theorem, we know that $\mathrm{Im} B_0^* = (\mathrm{ker}(B_0))^\circ = V_0^\circ$. Hence, since $A \vect{u}-\vect{f} \in V_0^\circ$, there exists a $p \in M$ such that
\begin{align}
A \vect{u} + B_0^* p  = \vect{f}.
\end{align}
Since $B_0^*$ is injective, $p$ is also unique. Since $\vect{u} \in V_{\vect{g}}$, we have that $B\vect{u} = 0$. Hence (\ref{FOC1}) and (\ref{FOC2}) hold.

We will now show that (\ref{FOC3}) holds. We note that $C_\gamma$ is a convex subset of a linear space. For any given $\zeta, \eta \in C_\gamma$ and $t \in [0,1]$, we therefore have that $\zeta + t(\eta-\zeta) \in C_\gamma$. Since $(\vect{u},\rho)$ is a local minimizer, it follows that for each $\eta \in C_\gamma$, if $0 < t < r/\|\eta-\rho\|_{L^\infty(\Omega)}$, with $r$ as in (\ref{isolatedminr}), then
\begin{align}
\frac{1}{t} (J(\vect{u}, \rho + t(\eta-\rho)) - J(\vect{u},\rho)) \geq 0.
\end{align}
From Proposition \ref{prop:frechet}, we know that $J$ is Fr\'echet differentiable, and therefore also Gateaux differentiable, with respect to $\rho$. Hence, by taking the limit as $t \to 0$, we see that
\begin{align}
 c_{\vect{u}}(\rho, \eta-\rho) = \langle J'_\rho(\vect{u},\rho), \eta - \rho \rangle  \geq 0  \;\; \text{for all} \; \eta \in C_\gamma.
\end{align}
Therefore (\ref{FOC3}) holds.
\end{proof}

The following lemma will be used in the proof of the next proposition. 

\begin{lemma}
\label{lem:Eeps}
Consider a nonzero function $\eta \in C_\gamma$ and the measurable non-empty set $E \subset \subset \mathrm{supp}(\eta)$, where $\subset \subset$ denotes that the containment is compact and  $\mathrm{supp}$ denotes the support of a function, i.e.~$\eta > 0$ a.e.~in $E$. Then, there exists an $\epsilon' > 0$ such that, for all $\epsilon \in (0, \epsilon']$, there exists a set $E_\epsilon \subseteq E$, $|E_\epsilon|>0$ where $\eta > \epsilon$ a.e.~in $E_\epsilon$. 
\end{lemma}
\begin{proof}
For a contradiction, suppose that there exists no such $\epsilon'$ such that $E_{\epsilon'}$ exists. This implies that 
\begin{align}
\text{for all} \;\; n \geq 0, \;\; |E \backslash \hat{E}_n| = 0,
\label{Eeps1}
\end{align}
where $\hat{E}_n := \{0 \leq \eta \leq 1/n \;\; \text{a.e.~in} \;\; E\}$. We see that $\varnothing = E \backslash \hat{E}_1 \subseteq E \backslash \hat{E}_2  \subseteq  \cdots \subseteq E \backslash \hat{E}_n \subseteq \cdots$, i.e.~$E \backslash \hat{E}_n$ is  nondecreasing. Note that the limit of a nondecreasing sequence of sets $(A_n)$ can be defined as $\lim_{n \to \infty} A_n := \cup_{n \geq 1} A_n$.  By (\ref{Eeps1}) we note that
\begin{align}
\lim_{n \to \infty} |E \backslash \hat{E}_n| = 0.
\label{Eeps2}
\end{align}
Moreover,
\begin{align}
\begin{split}
\cup_{n=1}^\infty E \backslash \hat{E}_n &= \lim_{n \to \infty} E \backslash \{0 \leq \eta \leq 1/n \;\; \text{a.e.~in} \;\; E\}\\
& = E \backslash \{ \eta = 0 \;\; \text{a.e.~in} \;\; E\} = E \backslash \varnothing = E.
\end{split}
\label{Eeps3}
\end{align}
Now we see that
\begin{align}
0 < |E| = |\cup_{n=1}^\infty E \backslash \hat{E}_n|= |\lim_{n \to \infty} E \backslash \hat{E}_n| = \lim_{n \to \infty}  |E \backslash \hat{E}_n| = 0,
\label{Eeps4}
\end{align}
where the first equality follows from (\ref{Eeps3}),  the second equality follows from the definition of the limit of a nondecreasing sequence of sets,  the third equality follows from the continuity of the Lebesgue measure, and the fourth equality follows from (\ref{Eeps2}). (\ref{Eeps4}) is a contradiction and, therefore, such an $\epsilon' > 0$ must exist. By choosing $E_\epsilon = E_{\epsilon'}$ for all $0 < \epsilon \leq \epsilon'$, we conclude that the statement holds for all $\epsilon \in (0, \epsilon']$. 
\end{proof}

In the result that follows, we are required to distinguish between different types of global and local minimizers.
\begin{definition}[Strict minimizer]
Let $Z$ be a Banach space and suppose that $z_0 \in Z$ is a local or global minimizer of the functional $J :Z\to \mathbb{R}$. We say that $z_0$ is a strict minimizer if there exists  an open neighborhood $E \subset Z$ of $z_0$ such that $J(z_0) < J(z)$ for all $z \neq z_0$, $z \in E$.
\end{definition}
\begin{definition}[Isolated minimizer]
Let $Z$ be a Banach space and suppose that $z_0 \in Z$ is a local or global minimizer of the functional $J :Z\to \mathbb{R}$. We say that $z_0$ is isolated if there exists  an open neighborhood $E \subset Z$ of $z_0$ such that there are no other minimizers contained in $E$.
\end{definition}
\begin{remark}
If $z$ is an isolated minimizer, then it is also a strict minimizer. 
\end{remark}
The following proposition is a property of strict minimizers that will be useful for the numerical analysis of the finite element method.

\begin{prop}
\label{prop:rhosupport}
Suppose that $\Omega \subset \mathbb{R}^d$ is a Lipschitz domain,  $d \in \{2,3\}$, and $\alpha$ satisfies properties \ref{alpha1}--\ref{alpha4}. Further assume that the minimizer $(\vect{u},\rho) \in H^1_{\vect{g},\mathrm{div}}(\Omega)^d \times C_\gamma$ of (\ref{borrvallmin}) is  a strict minimizer. Then, $\mathrm{supp}(\rho) \subseteq U$, where $U:=\mathrm{supp}(\vect{u})$.
\end{prop}
\begin{proof}
By definition of  a strict  minimizer, there exists an $r>0$ such that, for all $(\vect{w},\eta) \in H^1_{\vect{g},\mathrm{div}}(\Omega)^d \times C_\gamma$, $(\vect{w},\eta) \neq (\vect{u},\rho)$ that satisfies
\begin{align*}
\| \vect{u} - \vect{w}\|_{H^1(\Omega)} + \|\rho -\eta \|_{L^\infty(\Omega)}  <  r,
\end{align*}
we have that $J(\vect{u},\rho) < J(\vect{w}, \eta)$. For a contradiction, suppose that there exists a set $E \subset \Omega$, $E \cap U = \varnothing$, of positive measure, where $\rho > 0$ a.e.~in $E$. By Lemma \ref{lem:Eeps}, there exists an $\epsilon \in(0,r)$ such that there exists a set $E_\epsilon \subseteq E$, $|E_\epsilon| > 0$ where $\rho > \epsilon$ a.e.~in $E_\epsilon$.
Define $\tilde{\rho}$ as
\begin{align}
\tilde{\rho} := 
\begin{cases}
\rho & \;\; \text{a.e.~in} \;\; \Omega \backslash E_\epsilon,\\
\rho - \epsilon & \;\; \text{a.e.~in} \;\; E_\epsilon.
\end{cases}
\end{align}
As $\rho \in C_\gamma$, also $\tilde{\rho} \in C_\gamma$. We note that $\|\rho - \tilde{\rho}\|_{L^\infty(\Omega)} = \|\epsilon \|_{L^\infty(E_\epsilon)} < r$ and, therefore, $(\vect{u},\tilde{\rho})$ lies inside the minimizing neighborhood of the $(\vect{u},\rho)$. However, $J(\vect{u},\tilde{\rho}) = J(\vect{u},\rho)$ as $\rho$ and $\tilde{\rho}$ only differ on the set $E_\epsilon$, but $\vect{u} = \vect{0}$ a.e.~in $E_\epsilon \subseteq E$ by assumption. This contradicts the assertion that $(\vect{u},\rho)$ is  a strict  minimizer. 
\end{proof}

\section{Regularity of $\rho$}
\label{sec:regularity}

In  this  section we show that $\rho \in C_\gamma$ possesses higher regularity in the case of a homogeneous Dirichlet boundary condition on $\vect{u}$ and if $\alpha$ satisfies a stronger (but not restrictive) convexity assumption. 

{
\begin{theorem}[Regularity  of $\rho$]
\label{th:regularity1}
Suppose that the domain $\Omega \subset \mathbb{R}^d$ is bounded, the boundary is Lipschitz, and that the data $\vect{g} = \vect{0}$ on $\partial \Omega$. Consider a local or global minimizer, $(\vect{u}, \rho) \in H^1_{0, \mathrm{div}}(\Omega)^d \times C_\gamma$, of (\ref{borrvallmin}) such that $\vect{u}$ is not the  zero function  and there exists a closed subset $ \bar U_\theta  \subset \Omega$ with non-empty interior on which $|\vect{u}|^2$ is bounded below by a positive constant, $\theta > 0$. Suppose that \ref{alpha1}--\ref{alpha4} hold and that $\alpha \in C^2([0,1])$ is strongly convex, i.e.,
\begin{enumerate}[label=({A}\arabic*)]
\setlength\itemsep{0em}
\setcounter{enumi}{4}
\item There exists a constant $\alpha''_{\mathrm{min}}>0$ such that $\alpha ''(y) \geq \alpha''_{\mathrm{min}} > 0$ for all $y \in [0,1]$. \label{alpha5}
\end{enumerate}
 Consider the (non-empty) open interior $U_\theta \subset \bar{U}_\theta$.  Then, $\nabla \rho$ exists in $U_\theta$ and $\rho \in C_\gamma \cap H^1(U_\theta)$. 
\end{theorem}

\begin{remark}
The assumption \ref{alpha5} excludes the case where $\alpha$ is linear. This is consistent with previous theory, as Borrvall and Petersson \cite{Borrvall2003} showed that if $\alpha$ is linear,  then there exists a minimizer $(\vect{u}, \rho)$ where $\rho$ is a 0-1 solution (a linear combination of Heaviside functions) and thus $\rho \notin H^1(\Omega)$ due to the jumps.  However, the assumptions \ref{alpha1}--\ref{alpha5} do include (\ref{eq:alphachoice}), where the lower bound in \ref{alpha5} is $\alpha_{\mathrm{min}}''= 2\bar\alpha q/(q+1)^2$.  We see that this lower bound degrades to zero as $q \to \infty$.  As previously noted, the limit $q \to \infty$ coincides with $\alpha(\rho) \to  \bar{\alpha}(1-\rho)$, which is a linear function.  
\end{remark}

\begin{proof}[Proof of Theorem \ref{th:regularity1}]
 Let $\partial_{x_k}$ denote the partial derivative with respect to $x_k$.  If we can bound the $L^2$-norm of the difference quotients of $\rho$, in all coordinate directions in $U_\theta$, above by constants independent of $h$, then, by taking the weak limit, we deduce that $\partial_{x_k} \rho$ exists as an element of $L^2(U_\theta)$ for $1 \leq k \leq d$.

The variational inequality on $\rho$ states that
\begin{align}
\frac{1}{2} \int_\Omega \alpha'(\rho)|\vect{u}|^2 (\eta-\rho) \; \dx \geq 0. \label{regVI}
\end{align}
We define $U \subset \Omega$ as $U:=\text{supp}(\vect{u})$ and fix an open, bounded and connected domain $\hat \Omega$ such that $\hat \Omega = \Omega$ if $U \subset \subset  \Omega$ and $\Omega \subset \subset \hat \Omega$ otherwise. In the case where $U$ is not a compact subset of $\Omega$, we extend $\vect{u}$ and $\rho$ by zero to the whole of $\mathbb{R}^d$. Since the trace of $\vect{u}$ is zero on the boundary, the extension of $\vect{u}$ by zero lives in $H^1(\hat \Omega)^d$. Let $0<|h|<(1/2)\mathrm{dist}(U,\partial \hat\Omega)$ and choose $k \in \{1, \dots, d\}$. We define $\rho^h$ as 
\begin{align*}
\rho^h(x) = 
\begin{cases}
\rho(x+he_k) & \text{for} \; x \in \Omega-he_k,\\
0  &\text{for} \; x \in \mathbb{R}^d \backslash (\Omega-he_k).
\end{cases}
\end{align*}
 Analogously, we define $\vect{u}^h$ as
\begin{align*}
\vect{u}^h(x) = 
\begin{cases}
\vect{u}(x+he_k) & \text{for} \; x \in \Omega-he_k,\\
\vect{0}  &\text{for} \; x \in \mathbb{R}^d \backslash (\Omega-he_k).
\end{cases}
\end{align*}
 We define the difference quotient, $D^h_k$, in the $k$-th coordinate direction, as
\begin{align*}
D^h_k  \rho(x) = \frac{ \rho(x+he_k) - \rho(x)}{h}, \;\; h \in \mathbb{R}\backslash \{0\}, \; x \in  \hat\Omega.
\end{align*}
Let $\eta  =( \rho^h + \rho^{-h})/2$. We note that $\eta \in C_\gamma$, since
\begin{align*}
0 \leq \frac{1}{2} \rho^h \leq \frac{1}{2} \; \mathrm{a.e.\; in} \; \Omega, \quad \text{and} \quad
0 \leq \frac{1}{2} \rho^{-h} \leq \frac{1}{2} \; \mathrm{a.e.\; in} \; \Omega,
\end{align*}
which implies that $0 \leq \eta \leq 1$ a.e.\ in $\Omega$ and
\begin{align*}
\int_\Omega \eta \; \dx 
&= \frac{1}{2} \int_\Omega \rho^h + \rho^{-h} \; \dx \\
&= \frac{1}{2} \int_{\Omega -he_k \cap \Omega} \rho \;\dx+ \frac{1}{2} \int_{\Omega +he_k \cap \Omega} \rho \;\dx 
\leq \int_\Omega \rho \; \dx \leq \gamma |\Omega|.
\end{align*}
If we multiply (\ref{regVI}) through by 4 and divide by $h^2$ we see that
\begin{align}
\frac{1}{h^2}\int_{\Omega}\alpha'( \rho)| \vect{u}|^2 (  \rho^h + \rho^{-h} - 2  \rho) \; \dx \geq 0. \label{reg5}
\end{align}
We note that,
\begin{align*}
D^{-h}_k (D^h_k \rho) = \frac{\frac{\rho -  \rho^{-h}}{h} -\frac{ \rho^h - \rho}{h}}{-h} = \frac{\rho^h +  \rho^{-h} -2\rho }{h^2}.
\end{align*}
Hence, because $\vect{u}$ is zero outside of $\Omega$, (\ref{reg5}) is equivalent to
\begin{align}
\int_{ \hat\Omega}\alpha'( \rho)| \vect{u}|^2 (D^{-h}_k ( D^h_k  \rho)) \; \dx \geq 0. \label{reg3}
\end{align}
In order to obtain a first-order difference quotient, we will perform the finite difference analogue of integration by parts to shift the $D^{-h}_k$ operator from $D^h_k\rho$ to $\alpha'(\rho)|\vect{u}|^2$. We note that, by definition, the left-hand side of (\ref{reg3}) is equal to
\begin{align}
-\frac{1}{h} \int_{ \hat\Omega} (\alpha'( \rho)| \vect{u}|^2)(x)  \left((D^h_k  \rho)(x-he_k) - (D^h_k  \rho)(x) \right) \dx, \label{eq:regfd}
\end{align}
which by a change of variables is equal to
\begin{align*}
-\frac{1}{h} \left( \int_{ \hat\Omega-he_k} (\alpha'(\rho)| \vect{u}|^2)(x + he_k)  (D^h_k  \rho)(x)\dx -  \int_{ \hat\Omega} (\alpha'(\rho)| \vect{u}|^2)(x) (D^h_k \rho)(x) \dx\right).
\end{align*}
We note that $U \subset \subset  \hat \Omega$ and $|h| < (1/2)\text{dist}(U, \partial  \hat\Omega)$, which implies that $U \subset \subset \hat\Omega - he_k$. Therefore,
\begin{align}
\begin{split}
&\int_{ \hat \Omega-he_k} (\alpha'( \rho)| \vect{u}|^2)(x + he_k)  (D^h_k  \rho)(x)\dx \\
&\indent = \int_{U - he_k} (\alpha'(\rho)| \vect{u}|^2)(x + he_k)  (D^h_k  \rho)(x)\dx\\
&\indent \indent= \int_{ \hat\Omega} (\alpha'(\rho)| \vect{u}|^2)(x + he_k)  (D^h_k  \rho)(x)\dx. \label{reg4}
\end{split}
\end{align}
Therefore, from (\ref{reg3})--(\ref{reg4}) we see that
\begin{align}
\int_{ \hat\Omega}D^h_k (\alpha'( \rho)| \vect{u}|^2) ( D^h_k   \rho) \; \dx \leq  0. \label{reg6}
\end{align}
Now we wish to rewrite $D^h_k(\alpha'(\rho)|\vect{u}|^2)$ in a form that we can decouple from $D^h_k \rho$ in order to be able to bound (\ref{reg6}) above and below. Now,
\begin{align*}
D^h_k (\alpha'( \rho)| \vect{u}|^2)(x) &= \frac{1}{h}\left(\alpha'( \rho(x+he_k))| \vect{u}(x+he_k)|^2 - \alpha'( \rho(x))| \vect{u}(x)|^2\right)\\
&= \frac{1}{2h}\left(\alpha'( \rho(x+he_k)) \left(|\vect{u}(x+he_k)|^2 - | \vect{u}(x)|^2 \right)\right)\\
& \indent + \frac{1}{2h}\left(\alpha'( \rho(x)) \left(|\vect{u}(x+he_k)|^2 - | \vect{u}(x)|^2 \right)\right)\\
& \indent + \frac{1}{2h}\left( | \vect{u}(x+he_k)|^2 \left( \alpha'( \rho(x+he_k)) - \alpha'( \rho(x)) \right) \right)\\
& \indent + \frac{1}{2h}\left( | \vect{u}(x)|^2 \left( \alpha'( \rho(x+he_k)) - \alpha'( \rho(x)) \right) \right)\\
& = \frac{1}{2}\left(\alpha'( \rho^h) + \alpha'( \rho) \right)D^h_k(|\vect{u}|^2) + \frac{1}{2}\left(|\vect{u}^h|^2+ |\vect{u}|^2\right) D^h_k (\alpha'( \rho)).
\end{align*}
Therefore, from (\ref{reg6}) we see that
\begin{align}
\begin{split}
&\int_{ \hat\Omega} \left[ \frac{1}{2} \left(| \vect{u}^h|^2 +|\vect{u}|^2 \right) D^h_k(\alpha'(  \rho) ) \right.\\
&\indent \indent \indent \left.+ \frac{1}{2} \left(\alpha'(  \rho^h)+\alpha'(  \rho) \right) D^h_k | \vect{u}|^2 \right]  D^h_k  ( \rho )\; \dx \leq   0. \label{reg7}
\end{split}
\end{align}
Now,
\begin{align*}
&\frac{1}{2} \left(| \vect{u}^h|^2 +|\vect{u}|^2 \right) D^h_k(\alpha'(  \rho) ) + \frac{1}{2}\left( \alpha'(  \rho^h)+\alpha'(  \rho) \right) D^h_k | \vect{u}|^2  \\
&\indent = \frac{1}{h} \int_0^1 \frac{\text{d}}{\text{d}s} \left[ \alpha'\left(s  \rho^h + (1-s) \rho\right) \frac{1}{2} \left(| \vect{u}^h|^2 +|\vect{u}|^2 \right)  \right.\\
&\indent \indent + \left. \frac{1}{2}\left( \alpha'(  \rho^h)+\alpha'(  \rho) \right) \left| s  \vect{u}^h + (1-s)  \vect{u}\right|^2 \right] \text{d}s\\
&\indent = \frac{1}{h} \underbrace{\int_0^1 \left[ \alpha''\left(s   \rho^h + (1-s)  \rho\right) \right] \text{d}s}_{=:A}  \frac{1}{2} \left(| \vect{u}^h|^2 +|\vect{u}|^2 \right)  (  \rho^h- \rho)\\
& \indent \indent + \frac{1}{2h} \left( \alpha'(  \rho^h)+\alpha'(  \rho) \right)\underbrace{\int_0^1 \left[ 2\left( s  \vect{u}^h + (1-s)  \vect{u}\right)  \right] \text{d}s}_{=:\vect{B}} \cdot \; ( \vect{u}^h -  \vect{u}).
\end{align*}
Hence from (\ref{reg7}) we find that
\begin{align}
\frac{1}{2}\int_{ \hat\Omega} A(|\vect{u}^h|^2 +|\vect{u}|^2) |D^h_k  \rho|^2 +( \alpha'( \rho^h) + \alpha'( \rho))\vect{B} \cdot (D^h_k \vect{u})  D^h_k \rho \; \dx  \leq  0. \label{reg8}
\end{align}
Subtracting the second term on the left-hand side in (\ref{reg8}) from both sides, taking absolute values on the right-hand side, using the Cauchy--Schwarz inequality and multiplying by 2, we see that
\begin{align}
\int_{ \hat\Omega} A(|\vect{u}^h|^2 +|\vect{u}|^2)  |D^h_k  \rho|^2  \dx &\leq \int_{ \hat\Omega}  |\vect{B}| |\alpha'( \rho^h) + \alpha'( \rho) | |D^h_k  \vect{u}| |D^h_k  \rho| \dx. \label{reg9}
\end{align}
Furthermore we note that $A \geq \alpha''_{\mathrm{min}}$ and
\begin{align*}
\vect{B} =  \int_0^1 \left[ 2\left( s  \vect{u}^h + (1-s)  \vect{u}\right)  \right] \text{d}s = 2\left[ \frac{s^2}{2} \vect{u}^h + \left( s- \frac{s^2}{2}\right)  \vect{u} \right]^1_0 = \vect{u}^h + \vect{u}.
\end{align*}
Hence, using Cauchy's inequality and Young's inequality, we see that 
\begin{align}
\begin{split}
&\alpha''_{\mathrm{min}} \int_{U-he_k} |\vect{u}^h|^2|D^h_k \rho|^2  \dx + \alpha''_{\mathrm{min}} \int_{U} |\vect{u}|^2 |D^h_k \rho|^2  \dx\\
&\leq \int_{ \hat\Omega}  |\vect{u}+\vect{u}^h| |\alpha'( \rho^h) + \alpha'( \rho)| |D^h_k  \vect{u}| |D^h_k \rho| \dx\\
&\leq \int_{U-he_k}  |\vect{u}^h| |\alpha'( \rho^h) + \alpha'( \rho)| |D^h_k  \vect{u}| |D^h_k \rho| \dx \\
&\indent + \int_{U}  |\vect{u}| |\alpha'( \rho^h) + \alpha'( \rho)| |D^h_k  \vect{u}| |D^h_k \rho| \dx\\
& \leq  \frac{\epsilon}{2} \int_{U-he_k}  |\vect{u}^h|^2  |D^h_k \rho|^2 \dx +  \frac{\epsilon}{2} \int_{U}  |\vect{u}|^2  |D^h_k \rho|^2 \dx\\
&\indent + \frac{1}{2 \epsilon}\int_{U-he_k} |\alpha'(\rho^h)+\alpha'(\rho)|^2 |D^h_k  \vect{u}|^2 \dx\\
&\indent \indent + \frac{1}{2 \epsilon}\int_{U} |\alpha'(\rho^h)+\alpha'(\rho)|^2 |D^h_k  \vect{u}|^2 \dx. \label{reg16}
\end{split}
\end{align}
By fixing $\epsilon = \alpha''_{\mathrm{min}}$, {from (\ref{reg16})} we see that,
\begin{align}
\begin{split}
& \frac{\alpha''_{\mathrm{min}}}{2} \int_{U}|\vect{u}|^2 |D^h_k \rho|^2  \dx\\
& \indent \leq   \frac{\alpha''_{\mathrm{min}}}{2}  \int_{U} |\vect{u}|^2 |D^h_k \rho|^2  \dx + \frac{\alpha''_{\mathrm{min}}}{2} \int_{U-he_k}  |\vect{u}^h|^2 |D^h_k \rho|^2  \dx\\
&\indent \leq  \frac{1}{\alpha''_{\mathrm{min}}}\int_{\hat \Omega} |\alpha'(\rho^h)+ \alpha'( \rho)|^2 |D^h_k  \vect{u}|^2 \dx. \label{reg17}
\end{split}
\end{align}

Now $|\alpha'(\rho^h)+\alpha'(\rho)|^2$  is bounded above by $4\sup_{\zeta \in C_\gamma}|\alpha'(\zeta)|^2$ which is independent of $h$. Consider a set $\tilde \Omega \subset \mathbb{R}^d$ such that $\hat \Omega \subset \subset \tilde \Omega$. We note that $\vect{u} \in H^1(\tilde\Omega)^d$. By applying Theorem 3 in \cite[pg.~294]{Evans2010}, we see that
\begin{align}
\begin{split}
 \int_{U} |\vect{u}|^2 |D^h_k  \rho|^2 \; \dx &\leq \frac{\tilde C(\Omega)\sup_{\zeta \in C_\gamma}|\alpha'(\zeta)|^2}{(\alpha''_{\mathrm{min}})^2 }   \| \nabla \vect{u} \|_{L^2(\tilde \Omega)}^2\\
 & \leq \frac{\hat C(\Omega)\sup_{\zeta \in C_\gamma}|\alpha'(\zeta)|^2}{(\alpha''_{\mathrm{min}})^2 }   \| \nabla  \vect{u} \|_{L^2( \Omega)}^2 \leq C <\infty, \label{reg14}
 \end{split}
\end{align}
where  $\tilde C$, $\hat C$ and $C$ are constants. The bound is independent of $h$ and $k$.  Because, by hypothesis, there exists a subset $U_\theta \subset \Omega$ such that, $|\vect{u}|^2 \geq \theta > 0$ a.e.~in $U_\theta$, we see from (\ref{reg14}) that, because $U_\theta \subset U = \mathrm{supp}(\vect{u})$, also
\begin{align}
\theta \int_{U_\theta} |D^h_k \rho|^2 \dx \leq  \int_{{U_\theta}} |\vect{u}|^2 |D^h_k  \rho|^2 \; \dx  \leq  \int_{U} |\vect{u}|^2 |D^h_k  \rho|^2 \; \dx \leq C. \label{reg15}
\end{align}
Estimate (\ref{reg15}) implies that
\begin{align}
\sup_h \|D^h_k \rho\|_{L^2({U_\theta})} <\infty. \label{reg18}
\end{align}
From (\ref{reg18}) we see that there exists a function $\eta_k \in L^2({U_\theta})$ and a subsequence $h_i \to 0$ such that,
\begin{align*}
D^{h_i}_k \rho \weak \eta_k \;\; \text{weakly in} \; L^2({U_\theta}).
\end{align*}
Finally, we wish to identify $\eta_k$ with $\partial_{x_k} \rho$. First choose any smooth and compactly supported function, $\phi \in C^\infty_c({U_\theta})$. We note that
\begin{align*}
\int_{U_\theta} \rho \; \partial_{x_k} \phi \; \dx \leq C \|\rho\|_{L^\infty({U_\theta})} \|\phi\|_{W^{1,\infty}({U_\theta})} < \infty.
\end{align*}
Since $\partial_{x_k} \phi$ is compactly supported in ${U_\theta}$, it follows that
\begin{align*}
\int_{U_\theta} \rho \; \partial_{x_k} \phi \; \dx &= \int_{\hat \Omega}\rho \; \partial_{x_k} \phi \; \dx.
\end{align*}
Hence
\begin{align*}
\int_{U_\theta} \rho \; \partial_{x_k} \phi \; \dx & = \lim_{h_i \to 0}  \int_{\hat \Omega}\rho \; D^{-h_i}_k \phi \; \dx\\
&= -\lim_{h_i \to 0} \int_{\hat\Omega} (D^{h_i}_k \rho) \phi \; \dx
= -\lim_{h_i \to 0} \int_{{U_\theta}} (D^{h_i}_k \rho) \phi \; \dx
 = - \int_{U_\theta} \eta_k \phi \; \dx.
\end{align*}
Hence $\eta_k = \partial_{x_k} \rho$ a.e.\ in ${U_\theta}$ for $k = 1, \dots, d$. Therefore, from  (\ref{reg15})  we see by weak lower semicontinuity that
\begin{align}
\int_{U_\theta} \left| \partial_{x_k} \rho \right|^2 \dx \leq C(\Omega, \sup_{\zeta \in C_\gamma}|\alpha'(\zeta)|^2,  \alpha''_{\mathrm{min}},  \theta ), \label{reg19}
\end{align}
for some constant $C$. We conclude that $\rho \in H^1({U_\theta}) \cap C_\gamma$, $\theta > 0$.
\end{proof}}

\begin{remark}
The assumption that the boundary datum $\vect{g} = \vect{0}$ on $\partial \Omega$ is required so as to expand the domain of integration from $\Omega$ in (\ref{reg5}) to $\hat \Omega$ in (\ref{reg3}) in order to perform the finite difference analogue of integration by parts in (\ref{eq:regfd}). A homogeneous Dirichlet boundary condition on $\vect{u}$ is rarely imposed in practice. However, we observe during numerical experiments that $\rho$ possesses additional regularity in the case of inhomogeneous Dirichlet boundary conditions, and we hypothesize that the results can be generalized to that case. 
\end{remark}

\section[FEM approximation]{Finite element approximation}
\label{sec:femapprox}
We will be approximately solving (\ref{FOC1})--(\ref{FOC3}) by approximating the solutions  of the infinite-dimensional problem  with finite element functions. Borrvall and Petersson \cite[Sec.~3.3]{Borrvall2003} considered a piecewise constant finite element approximation of the material distribution coupled with an inf-sup stable quadrilateral finite element approximation of the velocity and the pressure. They showed that such approximations of the velocity and material distribution converge to an unspecified solution $(\vect{u}, \rho)$ of (\ref{borrvallmin}) in the following sense \cite{Borrvall2003}:
\begin{align*}
\vect{u}_h &\weak \vect{u} \;\; \text{weakly in} \;\; H^1(\Omega)^d, \\
\rho_h &\weakstar \rho \;\;  \text{weakly-* in} \;\; L^\infty(\Omega),\\
\rho_h &\to \rho \;\;  \text{strongly in} \;\; L^s(\Omega_b), \;\; s \in [1,\infty),
\end{align*}
where $\Omega_b$ is any measurable subset of $\Omega$ where $\rho$ is equal to zero or one a.e. There are no proven convergence results for the finite element approximation of the pressure, $p$. Since there can be multiple local minimizers, it is unclear which solution the sequence of finite element solutions is converging to. In this section, we consider any suitable conforming mixed finite element space such that the velocity and pressure spaces are inf-sup stable. We prove that, for \emph{every} isolated minimizer of (\ref{borrvallmin}), there exist sequences of finite element solutions to the discretized first-order optimality conditions that strongly converge to the minimizer as the mesh size tends to zero. More specifically, we show that, for each isolated minimizer  of the infinite-dimensional problem,  there exist (possibly different) sequences of finite element solutions $(\vect{u}_h, \rho_h, p_h)$ that converge to it strongly in $H^1(\Omega)^d \times L^s(\Omega) \times L^2(\Omega)$ as $h \to 0$, where $s \in [1,\infty)$. We emphasize that the results hold in the case where the local minima are isolated.

Consider the conforming finite element spaces $X_h \subset H^1(\Omega)^d$,  $C_{\gamma,h} \subset C_\gamma$,  and $M_h \subset L^2_0(\Omega)$. Let $X_{0,h} :=\{\vect{v}_h \in X_h : \vect{v}_h|_{\partial \Omega} = \vect{0}\}$.

In general, it will not be possible to represent the boundary data $\vect{g}$ exactly in the velocity finite element space. Hence, for each $h$, we  instead consider boundary data $\vect{g}_h$ (which can be represented) and assume that
\begin{enumerate}[label=({F}\arabic*)]
\item $\vect{g}_h \to \vect{g}$ strongly in $H^{1/2}(\partial \Omega)^d$. \label{ass:gh}
\end{enumerate}
We now define the space $X_{\vect{g}_h,h} :=\{\vect{v}_h \in X_h : \vect{v}_h|_{\partial \Omega} = \vect{g}_h\}$. We will also assume that:
\begin{enumerate}[label=({F}\arabic*)]
\setcounter{enumi}{1}
\item $X_{0,h}$ and $M_h$ satisfy the following inf-sup condition for some $c_b > 0$,  independent of $h$, 
\begin{align}
c_b \leq \inf_{q_h \in M_h \backslash \{0\}} \sup_{\vect{v}_h \in X_{0,h} \backslash \{0\}} \frac{b(\vect{v}_h,q_h)}{\|\vect{v}_h\|_{H^1(\Omega)}\|q_h\|_{L^2(\Omega)}}.
\end{align}
\label{ass:discreteinfsup}
\item The finite element spaces are dense in their respective function spaces, i.e., for any $(\vect{v}, \eta, q) \in H^1(\Omega)^d \times C_\gamma \times L^2_0(\Omega)$, \label{ass:dense}
\begin{align*}
\lim_{h \to 0} \inf_{\vect{w}_h \in X_{h}}\|\vect{v}-\vect{w}_h\|_{H^1(\Omega)} &= \lim_{h\to 0} \inf_{\zeta_h \in C_{\gamma,h}} \|\eta - \zeta_h \|_{L^2(\Omega)} \\
&\indent = \lim_{h\to 0} \inf_{r_h \in M_h} \|q - r_h \|_{L^2(\Omega)} = 0.
\end{align*}
\end{enumerate}

\begin{theorem}[Convergence of the finite element method]
\label{th:FEMexistence}
Let $\Omega \subset \mathbb{R}^d$ be a polygonal domain in two dimensions or a polyhedral  Lipschitz domain in three dimensions. Suppose that \ref{alpha1}--\ref{alpha5} hold and there exists an isolated local minimizer $(\vect{u},\rho) \in H^1_{\vect{g},\mathrm{div}}(\Omega)^d \times C_\gamma$ of (\ref{borrvallmin}). Moreover, assume that, for $\theta > 0$, $\bar U_\theta$ is the subset of $\Omega$ where $|\vect{u}|^2 \geq \theta$ a.e.~in $\bar U_\theta$ and suppose that there exists a $\theta'>0$ such that $\bar U_\theta$ is closed and has non-empty interior $U_\theta$ for all $\theta\leq\theta'$. Let $p$ denote the unique Lagrange multiplier associated with $(\vect{u},\rho)$ such that $(\vect{u},\rho,p)$ satisfy the first-order optimality conditions (\ref{FOC1})--(\ref{FOC3}). 
 
Consider the conforming finite element spaces $X_h \subset H^1(\Omega)^d$,  $C_{\gamma,h} \subset C_\gamma$,  and $M_h \subset L^2_0(\Omega)$ and suppose that the assumptions \ref{ass:gh}--\ref{ass:dense} hold. 

Then, there exists an $\bar h >0$ such that, for $\bar h \geq h \to 0 $, there is a sequence of solutions $(\vect{u}_{h}, \rho_{h}, p_{h}) \in X_{\vect{g}_h,{h}} \times C_{\gamma,{h}} \times M_{h}$  to the following discretized first-order optimality conditions
\begin{align}
a_{\rho_{h}}(\vect{u}_{h},\vect{v}_{h}) + b(\vect{v}_{h}, p_{h}) &= l_{\vect{f}}(\vect{v}_{h}) && \text{for all} \; \vect{v}_{h} \in X_{0,{h}}, \tag{FOC1$_h$} \label{FOC1h}\\
b(\vect{u}_{h},q_{h}) &=0 && \text{for all} \; q_{h} \in M_{h}, \tag{FOC2$_h$} \label{FOC2h}\\
c_{\vect{u}_{h}}(\rho_{h}, \eta_{h} - \rho_{h}) &\geq 0 && \text{for all} \; \eta_{h} \in C_{\gamma,{h}}, \tag{FOC3$_h$} \label{FOC3h}
\end{align}
such that $\vect{u}_{h} \to \vect{u}$ strongly in $H^1(\Omega)^d$, $\rho_{h} \to \rho$ strongly in $L^s(\Omega)$, $s \in [1,\infty)$,  and $p_{h} \to p$ strongly in $L^2(\Omega)$ as $h \to 0$.
\end{theorem}

In Proposition \ref{prop:FEMconvergence}, by fixing a ball around an isolated local minimizer, we show that finite element minimizers of a modified optimization problem converge weakly in $H^1(\Omega)^d \times L^2(\Omega)$ to the isolated minimizer  of the infinite-dimensional problem.   From this we deduce that there exists a subsequence of finite element solutions $(\vect{u}_h)$ that converges strongly to the isolated minimizer  of the infinite-dimensional problem  in $L^2(\Omega)^d$. We then strengthen the convergence of $\rho_h$ to strong convergence in $L^s(\Omega)$, $s \in [1,\infty)$, in Proposition \ref{prop:FEMconvergence1.5} and strengthen the convergence of $\vect{u}_h$ to strong convergence in $H^1(\Omega)^d$ in Proposition \ref{prop:FEMconvergence2}. In Proposition \ref{prop:FEMconvergence3}, we prove that there exists an $\bar h >0$ such that there is a subsequence, $\bar h > h\to 0 $, of strongly converging finite element solutions that also satisfy discretized first-order optimality conditions. Finally, in Proposition \ref{prop:FEMconvergence4}, we show that the Lagrange multiplier, $p_h \in M_h$, that satisfies the discretized first-order optimality conditions, converges strongly in $L^2(\Omega)$ to the Lagrange multiplier  for the infinite-dimensional problem. 

We now fix an isolated minimizer $(\vect{u},\rho)$ of (\ref{borrvallmin}). We define the radius of the basin of attraction as the largest value $r$ such that $(\vect{u},\rho)$ is the unique local minimizer in $B_{r, H^1(\Omega) \times L^2(\Omega)}(\vect{u},\rho) \cap (H^1_{\vect{g},\mathrm{div}}(\Omega)^d \times C_{\gamma})$, where
\begin{align}
\begin{split}
&B_{r, H^1(\Omega) \times L^2(\Omega)}(\vect{u},\rho) \\
&:= \{\vect{v} \in H^1(\Omega)^d, \; \eta \in C_\gamma: \|\vect{u} - \vect{v}\|_{H^1(\Omega)} + \|\rho - \eta\|_{L^2(\Omega)}\leq r \}.
\end{split}
\end{align}
We also define $B_{r,H^1(\Omega)}(\vect{u})$ and $B_{r,L^2(\Omega)}(\rho)$ by
\begin{align}
B_{r, H^1(\Omega)}(\vect{u}) &:= \{\vect{v} \in H^1(\Omega)^d : \|\vect{u} - \vect{v}\|_{H^1(\Omega)} \leq r \},\\
B_{r,L^2(\Omega)}(\rho)& := \{ \eta \in C_\gamma:\|\rho - \eta\|_{L^2(\Omega)} \leq r \}.
\end{align}
We note that 
\begin{align*}
(H^1_{\vect{g},\mathrm{div}}(\Omega)^d \cap B_{r/2,H^1(\Omega)}(\vect{u})) &\times (C_{\gamma}\cap B_{r/2,L^2(\Omega)}(\rho)) \\
&\indent \subset B_{r, H^1(\Omega) \times L^2(\Omega)}(\vect{u},\rho) \cap (H^1_{\vect{g},\mathrm{div}}(\Omega)^d \times C_{\gamma} )
\end{align*}
and hence  $(\vect{u},\rho)$ is also the unique minimizer in $(H^1_{\vect{g},\mathrm{div}}(\Omega)^d \cap B_{r/2,H^1(\Omega)}(\vect{u})) \times (C_{\gamma}  \cap B_{r/2,L^2(\Omega)}(\rho))$. 

Moreover, we define the spaces $V_{\vect{g}_h,h}$ and $V_{0,h}$ by
\begin{align*}
V_{\vect{g}_h,h} &:= \{ \vect{v}_h \in X_{\vect{g}_h,h} : b(\vect{v}_h,q_h) = 0 \; \text{for all}\; q_h \in M_h \},\\
V_{0,h} &:= \{ \vect{v}_h \in X_{0,h} : b(\vect{v}_h,q_h) = 0 \; \text{for all}\; q_h \in M_h \}.
\end{align*}

\begin{remark}
\label{rem:conf:isolation1}
In the context of the material distribution, it might be a more natural choice to assume that $\rho$ is isolated with respect to the $L^\infty$-norm. Assuming $\rho$ is isolated with respect to the $L^2$-norm is a stronger isolation assumption, as it cannot be guaranteed that if $\eta \in C_\gamma$ lives in an isolated neighborhood with respect to the $L^2$-norm, $\eta \in B_{r,L^2(\Omega)}(\rho)$, then, there exists an $r_* > 0$, such that $\eta \in B_{r_*,L^\infty(\Omega)}(\rho)$ where $B_{r_*,L^\infty(\Omega)}(\rho)$ is an isolated neighborhood with respect to the $L^\infty$-norm. We make this stronger isolation assumption as simple and continuous functions are not dense in $L^\infty(\Omega)$, but are dense in $L^2(\Omega)$. This has implications in the assumption \ref{ass:dense} and, subsequently, in the remaining results. However, as far as we are aware, the $L^2$-isolation assumption is valid for all practical problems found in the literature, in particular it holds for both examples found in Section \ref{sec:numerical}.
\end{remark}
\begin{remark}
\label{rem:conf:isolation2}
We note that balls centred at $\rho \in C_\gamma$ are equivalent if measured against any $L^s$-norm for $s \in [1,\infty)$. More precisely, for any $s, q \in [1,\infty)$, if $\eta \in C_\gamma \cap B_{r,L^q(\Omega)}(\rho)$, there exists an $r_* > 0$, depending on $s \in [1,\infty)$, such that $\eta \in C_\gamma \cap B_{r_*,L^s(\Omega)}(\rho)$. Hence, the assumption that the material distribution is isolated with respect to $L^2$-norm is not a stronger assumption than being isolated with respect to the $L^s$-norm provided $s \in [1,\infty)$. 
\end{remark}

In Propositions \ref{prop:FEMconvergence}--\ref{prop:FEMconvergence3} and Corollary \ref{cor:Omegab},   we fix an isolated minimizer $(\vect{u},\rho)$ of (\ref{borrvallmin}) and suppose that the conditions of Theorem \ref{th:FEMexistence} hold.

\begin{prop}[Weak convergence of ($\vect{u}_h, \rho_h)$ in $H^1(\Omega)^d \times L^2(\Omega)$]
\label{prop:FEMconvergence}
Consider the finite-dimensional optimization problem: find $(\vect{u}_h,\rho_h)$ that minimizes
\begin{align}
\min_{(\vect{v}_h,\eta_h) \in (V_{\vect{g}_h,h} \cap B_{r/2,H^1(\Omega)}(\vect{u})) \times (C_{\gamma, h} \cap B_{r/2,L^2(\Omega)}(\rho))} J(\vect{v}_h, \eta_h). \tag{BP$_{h}$} \label{BPh}
\end{align}
Then, a minimizer $(\vect{u}_h,\rho_h)$ of (\ref{BPh}) exists and there exist subsequences (up to relabeling) such that
\begin{align}
\vect{u}_h  &\weak \vect{u} \; \text{weakly in} \; H^1(\Omega)^d,\\
\vect{u}_h &\to \vect{u} \; \text{strongly in} \; L^2(\Omega)^d,\\
\rho_h  &\weakstar \rho \; \text{weakly-* in} \; L^\infty(\Omega),\\
\rho_h  &\weak \rho \; \text{weakly in} \; L^s(\Omega),\; s \in [1,\infty).
\end{align}
\end{prop}

\begin{proof}
The functional $J$ is continuous and $(V_{\vect{g}_h,h} \cap B_{r/2,H^1(\Omega)}(\vect{u})) \times (C_{\gamma, h} \cap B_{r/2,L^2(\Omega)}(\rho))$ is a finite-dimensional, closed and bounded set.  Moreover, for sufficiently small $h$ it is non-empty.  Therefore, it is sequentially compact by the Heine--Borel theorem. Hence $J$ obtains its infimum in $(V_{\vect{g}_h,h} \cap B_{r/2,H^1(\Omega)}(\vect{u})) \times (C_{\gamma, h} \cap B_{r/2,L^2(\Omega)}(\rho))$ and therefore, a minimizer $(\vect{u}_h, \rho_h)$ exists. 

By a corollary of Kakutani's Theorem, if a Banach space is reflexive then every norm-closed, bounded and convex subset of the Banach space is weakly compact and thus, by the Eberlein--\v{S}mulian theorem, sequentially weakly compact. It can be checked that $H^1(\Omega)^d \cap B_{r/2,H^1(\Omega)}(\vect{u})$ and $C_\gamma \cap B_{r/2,L^2(\Omega)}(\rho)$ are norm-closed, bounded and convex subsets of the reflexive Banach spaces $H^1(\Omega)^d$ and $L^2(\Omega)$, respectively. Therefore,  $H^1(\Omega)^d \cap B_{r/2,H^1(\Omega)}(\vect{u})$ is weakly sequentially compact in $H^1(\Omega)^d$  and $C_\gamma \cap B_{r/2,L^2(\Omega)}(\rho)$ is weakly sequentially compact in $L^2(\Omega)$. 

Hence we can extract subsequences, $(\vect{u}_h)$ and $(\rho_h)$ of the sequence generated by the global minimizers of (\ref{BPh}) such that
\begin{align}
\vect{u}_h &\weak \hat{\vect{u}} \in H^1(\Omega)^d \cap B_{r/2,H^1(\Omega)}(\vect{u}) \; \text{weakly in} \; H^1(\Omega)^d,\\
\rho_h &\weak \hat{\rho} \in C_\gamma \cap B_{r/2,L^2(\Omega)}(\rho) \; \text{weakly in} \; L^2(\Omega).
\end{align}

By assumption \ref{ass:dense}, there exists a sequence of finite element functions $\tilde{\rho}_h \in C_{\gamma,h}$ that strongly converges to  $\rho$ in $L^2(\Omega)$. Moreover let $\vect{\tilde{u}}_h \in V_{\vect{g}_h,h}$ be a finite element function taken from the sequence of finite element functions that satisfy $\vect{\tilde{u}}_h \to \vect{u}$ strongly in $H^1(\Omega)^d$. Such a sequence is shown to exist in \cite[Lemma 3.1]{Borrvall2003}.

We now wish to identify the limits $\hat{\vect{u}}$ and $\hat{\rho}$. Consider the following bound:
\begin{align}
\label{femconvergence1}
\begin{split}
&|J(\vect{\tilde{u}}_h, \tilde{\rho}_h) - J(\vect{u},\rho)|\\
&\indent \leq \int_\Omega |(\alpha(\rho)-\alpha(\tilde{\rho}_h)) |\vect{u}|^2 | + | \alpha(\tilde{\rho}_h) (|\vect{u}|^2 - |\vect{\tilde{u}}_h|^2)| \; \dx\\
&\indent \indent + \int_\Omega \nu \left||\nabla \vect{u} |^2-|\nabla \vect{\tilde{u}}_h|^2 |  \right| + 2 |\vect{f} \cdot (\vect{u} - \vect{\tilde{u}}_h)| \; \dx\\
&\indent \leq L_\alpha  \|\vect{u}\|^2_{L^4(\Omega)}\| \tilde{\rho}_h - \rho\|_{L^2(\Omega)} \\
&\indent \indent+ \bar{\alpha} \|\vect{\tilde{u}}_h -\vect{u}\|_{L^2(\Omega)}( \|\vect{\tilde{u}}_h - \vect{u}\|_{L^2(\Omega)} + 2\|\vect{u}\|_{L^2(\Omega)})\\
& \indent \indent +\nu \|\vect{\tilde{u}}_h - \vect{u} \|_{H^1(\Omega)}( \|\vect{\tilde{u}}_h - \vect{u}\|_{H^1(\Omega)} + 2\|\vect{u}\|_{H^1(\Omega)}) \\
&\indent\indent+ 2 \|\vect{f}\|_{L^2(\Omega)}\|\vect{\tilde{u}}_h - \vect{u}\|_{L^2(\Omega)},
\end{split}
\end{align}
where $L_\alpha$ denotes the Lipschitz constant for $\alpha$. From (\ref{femconvergence1}) we see that 
\begin{align*}
J(\vect{\tilde{u}}_h, \tilde{\rho}_h) \to J(\vect{u},\rho) \;\; \text{as} \;\; h \to 0. 
\end{align*}
Furthermore, for sufficiently small $h >0$ we note that 
\begin{align*}
(\vect{\tilde{u}}_h, \tilde{\rho}_h) \in (V_{\vect{g}_h,h} \cap B_{r/2,H^1(\Omega)}(\vect{u})) \times (C_{\gamma,h} \cap B_{r/2,L^2(\Omega)}(\rho)).
\end{align*}
Therefore,
\begin{align}
J(\vect{u}_h,\rho_h) \leq J(\vect{\tilde{u}}_h, \tilde{\rho}_h).
\end{align}
By taking the limit as $h \to 0$ and utilizing the strong convergence of $\vect{\tilde{u}}_h$ and $\tilde{\rho}_h$ to $\vect{u}$ and $\rho$, respectively, we see that
\begin{align}
\lim_{h \to 0} J(\vect{u}_h, \rho_h) \leq J(\vect{u}, \rho). \label{eq:femapprox1}
\end{align}
\ref{ass:gh} implies that 
\begin{align}
\vect{u}_h|_{\partial \Omega} = \vect{g}_h \to \vect{g} \;\; \text{strongly in} \;\; H^{1/2}(\partial \Omega)^d.  \label{eq:bcscon}
\end{align}
By assumption \ref{ass:dense}, for every $q \in L^2_0(\Omega)$, there exists a sequence of $\tilde{q}_h \in M_h$ such that $\tilde{q}_h \to q$ strongly in $L^2(\Omega)$. Since $\vect{u}_h \weak \hat{\vect{u}}$ weakly in $H^1(\Omega)^d$ and $\vect{u}_h \in V_{\vect{g}_h,h}$, we see that
\begin{align}
b(\hat{\vect{u}}, q) = \lim_{h \to 0} b(\vect{u}_h, \tilde{q}_h) +  \lim_{h \to 0} b(\vect{u}_h, q - \tilde{q}_h)= 0 \;\; \text{for all} \;\; q \in L^2_0(\Omega).
\end{align}
Hence $\hat{\vect{u}}$ is pointwise divergence-free and together with (\ref{eq:bcscon}), we deduce that $\hat{\vect{u}} \in  H^1_{\vect{g},\mathrm{div}}(\Omega)^d \cap B_{r/2,H^1(\Omega)}(\vect{u})$.  By construction $\hat{\rho} \in 
 C_\gamma \cap B_{r/2,L^2(\Omega)}(\rho)$. 
 
With a small modification to the proof found in \cite[Th.~3.1]{Borrvall2003}, we note that $J$ is weakly lower semicontinuous on $H^1(\Omega)^d \times L^2(\Omega)$. Therefore,
\begin{align}
J(\vect{\hat{u}},\hat{\rho}) \leq \liminf_{h \to 0} J(\vect{u}_h, \rho_h). \label{lowersemi}
\end{align}
We note that $(\vect{u},\rho)$ is the unique minimizer  of $ (H^1_{\vect{g},\mathrm{div}}(\Omega)^d \cap B_{r/2,H^1(\Omega)}(\vect{u})) \times (C_\gamma \cap B_{r/2,L^2(\Omega)}(\rho))$, which implies that $J(\vect{u},\rho) \leq J(\vect{\hat{u}},\hat{\rho})$. Hence, from (\ref{eq:femapprox1}) and (\ref{lowersemi}), it follows that
\begin{align}
J(\vect{\hat{u}},\hat{\rho}) = J(\vect{u}, \rho).
\end{align} 

Since $(\vect{u}, \rho)$ is the unique minimizer in the spaces we consider, we can identify the limits $\hat{\vect{u}}$ and $\hat{\rho}$ as $\vect{u}$ and $\rho$, respectively, and state that $\vect{u}_h \weak \vect{u}$ weakly in $H^1(\Omega)^d$ and $\rho_h \weak \rho$ weakly in $L^2(\Omega)$. By the Rellich--Kondrachov theorem, we can extract a further subsequence such that $\vect{u}_h \to \vect{u}$ strongly in $L^2(\Omega)^d$.

We note that by the Banach--Alaoglu theorem,  the closed unit ball of the dual space of a normed vector space, (for example $L^1(\Omega)$), is compact in the weak-* topology. Hence we can also find a subsequence such that $\rho_h \weakstar \hat{\rho} \in C_\gamma \cap \{\eta: \|\rho-\eta\|_{L^\infty(\Omega)} \leq r/2\}$ weakly-* in $L^\infty(\Omega)$. By the uniqueness of the weak limit, we can identify $\hat{\rho} = \rho$ a.e.~in $\Omega$ and, thus, we deduce that $\rho_h \weakstar \rho$ weakly-* in $L^\infty(\Omega)$,  i.e.~$\int_\Omega \rho_h \eta \, \mathrm{d}x \to \int_\Omega \rho \eta \, \mathrm{d}x$ for any $\eta \in L^1(\Omega)$. Since $L^s(\Omega) \subset L^1(\Omega)$, for any $s > 1$, we note that $\int_\Omega \rho_h \eta \, \mathrm{d}x \to \int_\Omega \rho \eta \, \mathrm{d}x$ for any $\eta \in L^s(\Omega)$, $s \geq 1$.  Hence, by definition,  $\rho_h \weak \rho$ weakly in $L^s(\Omega)$ for all $s \in [1,\infty)$. 
\end{proof}

\begin{corollary}[Strong convergence of $\rho_h$ in $L^s(\Omega_b)$ ]
\label{cor:Omegab}
Let $\Omega_b$ be any measurable subset of $\Omega$ of positive measure on which $\rho$ is equal to zero or one a.e.~(if such a set exists). Then, there exists a sequence of finite element minimizers, $\rho_h$, of (\ref{BPh}) that converge strongly in $L^s(\Omega_b)$ to the isolated local minimizer of (\ref{borrvallmin}), where $s \in [1,\infty)$.
\end{corollary}
\begin{proof}
We have shown that there exists a sequence of finite element minimizers $(\vect{u}_h, \rho_h)$ of (\ref{BPh}) that converge to the isolated local minimizer $(\vect{u}, \rho)$. In particular $\rho_h \weakstar \rho$ weakly-* in $L^\infty(\Omega)$ and $\rho_h \weak \rho$ weakly in $L^s(\Omega)$ for all $s \in [1,\infty)$. The result is then deduced by following the proof of Corollary 3.2 in \cite{Petersson1999}. 
\end{proof}

\begin{prop}[Strong convergence of $\rho_h$ in $L^s(\Omega)$, $s\in[1,\infty)$]
\label{prop:FEMconvergence1.5}
There exists a subsequence of minimizers, $(\rho_h)$, of (\ref{BPh}) such that
\begin{align}
\rho_h \to \rho \; \text{strongly in} \; L^s(\Omega), \;\; s \in [1,\infty).
\end{align}
\end{prop}
\begin{proof}
We note that $C_{\gamma,h} \cap B_{r/2,L^2(\Omega)}(\rho)$ is a convex set, and hence for any $\eta_h \in C_{\gamma,h} \cap B_{r/2,L^2(\Omega)}(\rho)$, $t \in [0,1]$, we have that $ \rho_h  + t(\eta_h - \rho_h) \in C_{\gamma,h} \cap B_{r/2,L^2(\Omega)}(\rho)$. Since $(\vect{u}_h, \rho_h)$ is a minimizer of (\ref{BPh}), by the arguments used in Proposition \ref{prop:pressureexistence} we deduce that
\begin{align}
\int_\Omega \alpha'(\rho_h) |\vect{u}_h|^2 (\eta_h - \rho_h) \dx \geq 0 \;\; \text{for all} \;\; \eta_h \in  C_{\gamma,h} \cap B_{r/2,L^2(\Omega)}(\rho). \label{femconvergence1.5-1}
\end{align}
Hence, (\ref{FOC3}) and (\ref{femconvergence1.5-1}) imply that for all $\eta \in C_\gamma$ and $\eta_h \in  C_{\gamma,h} \cap B_{r/2,L^2(\Omega)}(\rho)$ we have that
\begin{align}
\int_\Omega \alpha'(\rho) |\vect{u}|^2 \rho \; \dx &\leq \int_\Omega \alpha'(\rho) |\vect{u}|^2 \eta \; \dx, \label{materialerror1}\\
\int_\Omega \alpha'(\rho_h) |\vect{u}_h|^2 \rho_h \; \dx &\leq \int_\Omega \alpha'(\rho_h) |\vect{u}_h|^2 \eta_h \; \dx. \;\;\label{materialerror2}
\end{align}
By subtracting $\int_\Omega \alpha'(\rho) |\vect{u}|^2 \rho_h \dx$ from (\ref{materialerror1}) and $\int_\Omega \alpha'(\rho_h) |\vect{u}_h|^2 \rho \, \dx$ from (\ref{materialerror2}), we see that
\begin{align}
\int_\Omega \alpha'(\rho) |\vect{u}|^2 (\rho - \rho_h) \; \dx &\leq \int_\Omega \alpha'(\rho)  |\vect{u}|^2 (\eta - \rho_h)\dx, \label{materialerror3}\\
\int_\Omega \alpha'(\rho_h)  |\vect{u}_h|^2 (\rho_h - \rho)\; \dx &\leq \int_\Omega \alpha'(\rho_h)  |\vect{u}_h|^2 (\eta_h - \rho) \; \dx . \label{materialerror4}
\end{align}
Summing (\ref{materialerror3}) and (\ref{materialerror4}) and rearranging the left-hand side, we see that
\begin{align}
\begin{split}
&\int_\Omega (\alpha'(\rho) - \alpha'(\rho_h))|\vect{u}|^2 (\rho - \rho_h)\dx + \int_\Omega \alpha'(\rho_h) (|\vect{u}|^2 - |\vect{u}_h|^2)(\rho - \rho_h) \dx\\
&\indent \leq \int_\Omega \alpha'(\rho)  |\vect{u}|^2 (\eta - \rho_h)\dx+ \int_\Omega \alpha'(\rho_h) |\vect{u}_h|^2 (\eta_h - \rho) \dx. \label{materialerror5}
\end{split}
\end{align}
By fixing $\eta = \rho_h \in C_\gamma$ and subtracting the second term on the left-hand side of (\ref{materialerror5}) from both sides we deduce that
\begin{align}
\label{materialerror6}
\begin{split}
&\int_\Omega (\alpha'(\rho) - \alpha'(\rho_h))|\vect{u}|^2 (\rho - \rho_h)\dx\\
&\indent \leq \int_\Omega \alpha'(\rho_h) |\vect{u}_h|^2 (\eta_h - \rho) \dx 
 + \int_\Omega \alpha'(\rho_h) (|\vect{u}_h|^2-|\vect{u}|^2)(\rho - \rho_h) \dx.
\end{split}
\end{align}
By an application of the mean value theorem, we note that there exists a $c\in (0,1)$ such that
\begin{align}
\begin{split}
&\int_\Omega (\alpha'(\rho) - \alpha'(\rho_h)) |\vect{u}|^2 (\rho - \rho_h)\dx\\
& \indent = \int_\Omega \alpha''(\rho_h + c(\rho - \rho_h))  |\vect{u}|^2 (\rho - \rho_h)^2\dx.
\end{split}\label{materialerror11}
\end{align}
By \ref{alpha5} and the definition of $U_\theta$ we bound (\ref{materialerror11}) from below: 
\begin{align}
\label{materialerror7}
\begin{split}
 &\int_\Omega \alpha''(\rho_h + c(\rho - \rho_h)) |\vect{u}|^2 (\rho - \rho_h)^2\dx\\
 & \indent \geq  \int_{U _\theta} \alpha''(\rho_h + c(\rho - \rho_h))  |\vect{u}|^2 (\rho - \rho_h)^2\dx
\geq \alpha''_{\mathrm{min}} \theta \|\rho - \rho_h\|^2_{L^2(U_\theta)}.
\end{split}
\end{align}
Now we bound the right-hand side of (\ref{materialerror6}) as follows,
\begin{align}
\label{materialerror8}
\begin{split}
&\int_\Omega \alpha'(\rho_h) |\vect{u}_h|^2 (\eta_h - \rho) \dx 
 + \int_\Omega \alpha'(\rho_h) (|\vect{u}_h|^2-|\vect{u}|^2)(\rho - \rho_h) \dx \\
&\indent \leq 2\alpha'_{\text{max}}(\|\vect{u}\|^2_{L^4(\Omega)} + \|\vect{u} - \vect{u}_h\|^2_{L^4(\Omega)}) \| \rho - \eta_h \|_{L^2(\Omega)} \\
&\indent\indent+ \alpha'_{\text{max}} \|\rho-\rho_h\|_{L^q(\Omega)} \| \vect{u} + \vect{u}_h \|_{L^{q'}(\Omega)} \| \vect{u} - \vect{u}_h \|_{L^2(\Omega)},
\end{split}
\end{align}
where $2<q' < \infty$ in two dimensions, $2<q' \leq 6$ in three dimensions, and $q = 2q'/(q'-2)$. We note that 
\begin{align}
\begin{split}
&\| \vect{u} + \vect{u}_h \|_{L^{q'}(\Omega)} \leq \| \vect{u} \|_{L^{q'}(\Omega)}  + \| \vect{u}_h \|_{L^{q'}(\Omega)} \\
&\indent \leq \| \vect{u} \|_{H^1(\Omega)}  + \| \vect{u}_h \|_{H^1(\Omega)} \leq \hat{C} < \infty,
\end{split} \label{materialerror13}
\end{align}
where the second inequality holds thanks to the Sobolev embedding theorem. 
Combining (\ref{materialerror6})--(\ref{materialerror13}) we see that
\begin{align}
\label{materialerror9}
\|\rho - \rho_h\|^2_{L^2(U_\theta)} \leq C\left( \| \rho - \eta_h \|_{L^2(\Omega)} +  \|\rho-\rho_h\|_{L^q(\Omega)}\| \vect{u} - \vect{u}_h \|_{L^2(\Omega)} \right),
\end{align}
where $C = C(\alpha'_{\text{max}}, \alpha''_{\mathrm{min}}, \theta, \|\vect{u}\|_{L^4(\Omega)}, \hat{C})$. By assumption \ref{ass:dense}, there exists a sequence of finite element functions $\tilde{\rho}_h \in C_{\gamma,h}$ such that $\tilde{\rho}_h \to \rho$ strongly in $L^2(\Omega)$. Thanks to the strong convergence, we note that for sufficiently small $h$, $\tilde{\rho}_h \in C_{\gamma,h} \cap B_{r/2,L^2(\Omega)}(\rho)$. Hence we can fix $\eta_h =  \tilde{\rho}_h$. By Proposition \ref{prop:FEMconvergence}, we know that $\vect{u}_h \to \vect{u}$ strongly in $L^2(\Omega)^d$ and since $\rho \in C_\gamma$, $\rho_h \in C_{\gamma,h} \subset C_{\gamma}$, then $\| \rho - \rho_h \|_{L^q(\Omega)}\leq |\Omega|^{1/q} \| \rho - \rho_h \|_{L^\infty(\Omega)} \leq |\Omega|^{1/q}$. Therefore, the right-hand side of (\ref{materialerror9}) tends to zero as $h \to 0$. Hence, we deduce that 
\begin{align}
\rho_h \to \rho \;\; \text{strongly in} \;\; L^2(U_\theta), \;\; \theta > 0. 
\label{materialerror15}
\end{align}
Now we note that
\begin{align}
\label{rhoOmegaSplit}
\|\rho - \rho_h\|_{L^2(\Omega)}  = \|\rho - \rho_h\|_{L^2(U_\theta)} + \|\rho - \rho_h\|_{L^2(U\backslash U_\theta)} + \|\rho - \rho_h\|_{L^2(\Omega\backslash U)}.
\end{align}
If $U\backslash U_\theta$ or $\Omega \backslash U$ are empty, we neglect the corresponding term in (\ref{rhoOmegaSplit}) with no loss of generality. Suppose $\Omega \backslash U$ is non-empty.  By definition of $U$, $\vect{u} = \vect{0}$ a.e.~in $\Omega \backslash U$. By Proposition \ref{prop:rhosupport}, this implies that $\rho = 0$ a.e.~in $\Omega \backslash U$. Therefore, Corollary \ref{cor:Omegab} implies that 
\begin{align}
\label{materialerror16}
\rho_h \to \rho \;\;\text{strongly in} \;\; L^2(\Omega \backslash U). 
\end{align}
Suppose $U\backslash U_\theta$ is non-empty. Since, $\rho, \rho_h \in C_\gamma$ we see that 
\begin{align}
\label{materialerror17}
\|\rho - \rho_h\|_{L^2(U\backslash U_\theta)} \leq |U \backslash U_\theta|^{1/2} \to 0 \;\; \text{as} \;\; \theta \to 0.
\end{align} 
Therefore, by first taking the limit as $h \to 0$ and then by taking the limit as $\theta \to 0$, (\ref{materialerror15})--(\ref{materialerror17}) imply that $\rho_h \to \rho$ strongly in $L^2(\Omega)$. 

Since $\| \rho - \rho_h \|_{L^1(\Omega)} \leq |\Omega|^{1/2} \|\rho - \rho_h \|_{L^2(\Omega)}$, we see that $\rho_h \to \rho$ strongly in $L^1(\Omega)$. Hence, for any $s \in [1,\infty)$,
\begin{align}
\int_{\Omega} | \rho - \rho_h|^s \dx = \int_{\Omega}  | \rho - \rho_h|^{s-1}  | \rho - \rho_h|\dx \leq 1^{s-1} \| \rho - \rho_h \|_{L^1(\Omega)},
\end{align}
which implies that $\rho_h \to \rho$ strongly in $L^s(\Omega)$.
\end{proof}

\begin{prop}[Strong convergence of $\vect{u}_h$ in $H^1(\Omega)^d$]
\label{prop:FEMconvergence2}
There exists a subsequence of minimizers, $(\vect{u}_h)$, of (\ref{BPh}) such that
\begin{align}
\vect{u}_h \to \vect{u} \; \text{strongly in} \; H^1(\Omega)^d.
\end{align}
\end{prop}
\begin{proof}
We note that the set $W_h:=V_{\vect{g}_h,h} \cap B_{r/2,H^1(\Omega)}(\vect{u})$ is convex. Hence, by following the same arguments for deriving the variational inequality on $\rho$ as in Proposition \ref{prop:pressureexistence}, it can be shown that minimizers of (\ref{BPh}) satisfy the variational inequality
\begin{align}
a_{\rho_h}(\vect{u}_h,\vect{w}_h-\vect{u}_h) - l_f(\vect{w}_h-\vect{u}_h) \geq 0 \quad \text{for all} \quad \vect{w}_h \in W_{h}. \label{eq:vi-h}
\end{align}
We note that in the next proposition (Proposition \ref{prop:FEMconvergence3}), we will show that (\ref{eq:vi-h}) can be strengthened to an equality. However, this result is not currently available at this point and an equality does not follow from the arguments in Proposition \ref{prop:pressureexistence}. We note though that an inequality is sufficient for the subsequent arguments. From Proposition \ref{prop:pressureexistence} we deduce that, for all $\vect{w}_h \in W_{h}$,
\begin{align*}
a_{\rho}(\vect{u},\vect{w}_h-\vect{u}_h) + b(\vect{w}_h-\vect{u}_h,p) &= l_f(\vect{w}_h-\vect{u}_h).
\end{align*}
Hence
\begin{align}
a_{\rho_h}(\vect{u}_h,\vect{u}_h-\vect{w}_h) \leq a_{\rho}(\vect{u},\vect{u}_h-\vect{w}_h) + b(\vect{u}_h-\vect{w}_h,p).
\end{align}
By subtracting $a_{\rho_h}(\vect{w}_h,\vect{u}_h-\vect{w}_h)$ from both sides we see that
\begin{align*}
&a_{\rho_h}(\vect{u}_h-\vect{w}_h,\vect{u}_h-\vect{w}_h) \\
&\indent \leq a_{\rho}(\vect{u},\vect{u}_h-\vect{w}_h) - a_{\rho_h}(\vect{w}_h,\vect{u}_h-\vect{w}_h)  + b(\vect{u}_h-\vect{w}_h,p-q_h).
\end{align*}
We note that $a_{\rho_h}$  is coercive with constant $c_a  = \nu/(c_p^2 +1)$, and $b$ is bounded with constant $C_b$. Hence,
\begin{align*}
&\|\vect{u}_h - \vect{w}_h\|^2_{H^1(\Omega)} \leq \frac{1}{c_a} a_{\rho_h}(\vect{u}_h-\vect{w}_h,\vect{u}_h - \vect{w}_h) \\
&\indent\leq \frac{1}{c_a} \left(a_{\rho}(\vect{u},\vect{u}_h - \vect{w}_h) - a_{\rho_h}(\vect{w}_h,\vect{u}_h - \vect{w}_h)  + b(\vect{u}_h - \vect{w}_h,p-q_h)\right)\\
&\indent= \frac{1}{c_a}\left(\int_\Omega \alpha(\rho_h)(\vect{u} - \vect{w}_h) \cdot (\vect{u}_h - \vect{w}_h)
+ \left(\alpha(\rho) - \alpha(\rho_h)\right) \vect{u} \cdot (\vect{u}_h-\vect{w}_h) \dx\right.\\
&\indent\indent + \left.\int_\Omega \nu \nabla(\vect{u}- \vect{w}_h) : \nabla(\vect{u}_h - \vect{w}_h)  \dx + b(\vect{u}_h - \vect{w}_h,p-q_h)\right)\\
&\indent\leq \frac{1}{c_a} \bar \alpha \|\vect{u}-\vect{w}_h\|_{L^2(\Omega)}\|\vect{u}_h-\vect{w}_h\|_{L^2(\Omega)}\\
&\indent\indent+ \frac{1}{c_a} \|(\alpha(\rho) - \alpha(\rho_h)) \vect{u} \|_{L^{2}(\Omega)} \|\vect{u}_h-\vect{w}_h\|_{L^2(\Omega)}\\
&\indent\indent+ \frac{\nu}{c_a} |\vect{u} - \vect{w}_h|_{H^1(\Omega)}|\vect{u}_h-\vect{w}_h|_{H^1(\Omega)}
+ \frac{C_b}{c_a} \|\vect{u}_h-\vect{w}_h\|_{H^1(\Omega)}\|p-q_h\|_{L^2(\Omega)}.
\end{align*}
Hence,
\begin{align*}
\|\vect{u}_h - \vect{w}_h\|_{H^1(\Omega)} \leq C\left(\|\vect{u}-\vect{w}_h\|_{H^1(\Omega)}+ \|(\alpha(\rho) - \alpha(\rho_h)) \vect{u}\|_{L^{2}(\Omega)} + \|p-q_h\|_{L^2(\Omega)}\right),
\end{align*}
where $C = C(\bar\alpha, \nu, c_a, C_b)$  is a constant. This implies that, for all $\vect{w}_h \in W_h$,
\begin{align}
\begin{split}
&\|\vect{u} - \vect{u}_h\|_{H^1(\Omega)}\\
 &\indent\leq C'\left(\|\vect{u}-\vect{w}_h\|_{H^1(\Omega)}+ \|(\alpha(\rho) - \alpha(\rho_h)) \vect{u} \|_{L^{2}(\Omega)} + \|p-q_h\|_{L^2(\Omega)}\right). \label{uhh1convergence1}
\end{split}
\end{align}
where $C' = C'(\bar\alpha, \nu, c_a, C_b, L_\alpha)$. For sufficiently small $h$, we note that $\vect{\tilde{u}}_h \in W_h$ (where $\vect{\tilde{u}}_h$ is defined in the proof of Proposition \ref{prop:FEMconvergence}) and $\vect{\tilde{u}}_h \to \vect{u}$ strongly in $H^1(\Omega)^d$. Moreover by assumption \ref{ass:dense}, there exists a sequence of finite element functions $\tilde{p}_h \in M_h$ that converges to $p$ strongly in $L^2(\Omega)$. Suppose $\vect{w}_h =\vect{\tilde{u}}_h$ and $q_h = \tilde{p}_h$. From Proposition \ref{prop:FEMconvergence1.5}, we know that there exists a subsequence (not indicated) such that  $\rho_h \to \rho$ strongly in $L^4(\Omega)$. We now observe that
\begin{align}
\|(\alpha(\rho) - \alpha(\rho_h)) \vect{u} \|_{L^{2}(\Omega)} 
\leq L_\alpha  \|\rho - \rho_h\|_{L^4(\Omega)} \| \vect{u}\|_{L^{4}(\Omega)}
\end{align}
where $L_\alpha$ is the Lipschitz constant for $\alpha$. Hence by taking the limit as $h\to 0$, we deduce that $\vect{u}_h \to \vect{u}$ strongly in $H^1(\Omega)^d$. 
\end{proof}

In the following proposition, we show that (up to a subsequence) minimizers of (\ref{BPh}) also satisfy the first-order optimality conditions that are the finite-dimensional analogue of the first-order optimality conditions associated with (\ref{borrvallmin}). This allows us to consider the finite-dimensional optimization problem over the whole set $V_{\vect{g}_h,h} \times C_{\gamma,h}$, rather than the restricted set $(V_{\vect{g}_h,h} \cap B_{r/2,H^1(\Omega)}(\vect{u})) \times (C_{\gamma, h} \cap B_{r/2,L^2(\Omega)}(\rho))$.

\begin{prop}[Discretized first-order optimality conditions]
\label{prop:FEMconvergence3}
There exists an $\bar h > 0$ such that for all $h < \bar h$, there exists a unique Lagrange multiplier $p_h \in M_h$ such that the functions $(\vect{u}_h, \rho_h)$ that locally minimize (\ref{BPh}) satisfy the first-order optimality conditions (\ref{FOC1h})--(\ref{FOC3h})
\end{prop}
\begin{proof}
From Proposition \ref{prop:FEMconvergence2}, we know that $\vect{u}_h \to \vect{u}$ strongly in $H^1(\Omega)^d$. Hence by definition of strong convergence, there exists an $\bar h>0$ such that, for all $h \leq \bar h$, $\|\vect{u} - \vect{u}_{h}\|_{H^1(\Omega)} \leq r/4$. Therefore, for each $\vect{v}_h \in V_{0,h}$, if $|t|<r/(4\|\vect{v}_h\|_{H^1(\Omega)})$ then $\vect{u}_h + t \vect{v}_h \in V_{\vect{g}_h,h} \cap B_{r/2,H^1(\Omega)}(\vect{u})$. Now we can follow the reasoning of the proof of Proposition \ref{prop:pressureexistence} (adding the subscript $_h$ where necessary) to deduce the existence of a unique $p_h \in M_h$ such that (\ref{FOC1h})--(\ref{FOC3h}) hold.
\end{proof}
\begin{prop}[Strong convergence of $p_h$ in $L^2(\Omega)$]
\label{prop:FEMconvergence4}
There is a subsequence of the unique $p_h \in M_h$ defined in Proposition \ref{prop:FEMconvergence3} that converges strongly in $L^2(\Omega)$ to the $p\in L^2_0(\Omega)$ that solves (\ref{FOC1})--(\ref{FOC3}) for the given isolated local minimizer $(\vect{u}, \rho)$.
\end{prop}
\begin{proof}
The inf-sup condition \ref{ass:discreteinfsup} for $M_h$ and $X_{0,h}$ implies that, for any $q_h \in M_h$,
\begin{align*}
c_b \|q_h-p_h\|_{L^2(\Omega)}  &\leq \sup_{\vect{w}_h \in X_{0,h} \backslash \{0\}} \frac{b(\vect{w}_h,q_h-p_h)}{\|\vect{w}_h\|_{H^1(\Omega)}}\\
&=  \sup_{\vect{w}_h \in X_{0,h}  \backslash \{0\}} \frac{b(\vect{w}_h,p-p_h)+b(\vect{w}_h,q_h-p)}{\|\vect{w}_h\|_{H^1(\Omega)}}\\
&\leq  \sup_{\vect{w}_h \in X_{0,h}  \backslash \{0\}} \frac{|b(\vect{w}_h,p-p_h)|+|b(\vect{w}_h,q_h-p)|}{\|\vect{w}_h\|_{H^1(\Omega)}}\\
&= \sup_{\vect{w}_h \in X_{0,h} \backslash \{0\}} \frac{|a_\rho(\vect{u},\vect{w}_h) - a_{\rho_h}(\vect{u}_h,\vect{w}_h)|+|b(\vect{w}_h,q_h-p)|}{\|\vect{w}_h\|_{H^1(\Omega)}}\\
& \leq \|(\alpha(\rho) - \alpha(\rho_h))\vect{u}\|_{L^2(\Omega)} + \left(\bar \alpha+\nu\right)\|\vect{u}-\vect{u}_h\|_{H^1(\Omega)}\\
&\indent + C_b \|p-q_h\|_{L^2(\Omega)}.
\end{align*}
Hence,
\begin{align*}
\|p-p_h\|_{L^2(\Omega)} \leq C \left(\|(\alpha(\rho) - \alpha(\rho_h)) \vect{u} \|_{L^{2}(\Omega)} + \|\vect{u}-\vect{u}_h\|_{H^1(\Omega)} + \|p-q_h\|_{L^2(\Omega)}\right),
\end{align*}
where $C = C(c_b, C_b, \bar \alpha, L_\alpha, \nu)$. By assumption \ref{ass:dense}, there exists a sequence of finite element functions, $\tilde{p}_h \in M_h$ that satisfies $\tilde{p}_h \to p$ strongly in $L^2(\Omega)$. Let  $q_h = \tilde{p}_h$. We have already shown that $\vect{u}_h \to \vect{u}$ strongly in $H^1(\Omega)^d$ in Proposition \ref{prop:FEMconvergence2}. Similarly, in the proof of Proposition \ref{prop:FEMconvergence2} we also showed that $\|(\alpha(\rho) - \alpha(\rho_h)) \vect{u} \|_{L^{2}(\Omega)} \to 0$. Hence we conclude that $p_h \to p$ strongly in $L^2(\Omega)$.
\end{proof}

\begin{proof}[Proof of Theorem \ref{th:FEMexistence}]
Fix an isolated minimizer $(\vect{u}, \rho)$ of (\ref{borrvallmin}) and its unique associated Lagrange multiplier $p$. By the results of Propositions \ref{prop:FEMconvergence}, \ref{prop:FEMconvergence1.5}, \ref{prop:FEMconvergence2}, and \ref{prop:FEMconvergence3}, there exists a mesh size $\bar h$ such that for, $h < \bar h$, there exists a sequence of finite element solutions $(\vect{u}_h, \rho_h, p_h) \in V_{\vect{g}_h,h} \times C_{\gamma,h} \times M_h$ satisfying (\ref{FOC1h})--(\ref{FOC3h}) that converges to $(\vect{u}, \rho, p)$. By taking a subsequence if necessary (not indicated), Proposition \ref{prop:FEMconvergence1.5}, implies that $\rho_h \to \rho$ strongly in $L^s(\Omega)$, $s \in [1,\infty)$, Proposition \ref{prop:FEMconvergence2} implies that $\vect{u}_h \to \vect{u}$ strongly in $H^1(\Omega)^d$, and Proposition \ref{prop:FEMconvergence4} implies that $p_h \to p$ strongly in $L^2(\Omega)$. 
\end{proof}

\section{Numerical results}
\label{sec:numerical}
The main goal of this section is to experimentally verify the existence of strongly converging sequences that were proven to exist in Section \ref{sec:femapprox}. In all examples the systems are discretized with the finite element method using FEniCS \cite{fenics}. The computational domains are triangulated with simplices and we define the mesh size, $h$, as the maximum diameter of all the simplices in the triangulation. 
 The solutions are computed using the deflated barrier method \cite{Papadopoulos2021}. The deflated barrier method reformulates the discretized first-order optimality conditions (\ref{FOC1h})--(\ref{FOC3h}) as a mixed complementarity problem. The mixed complementarity problem is then solved with a primal-dual active set strategy \cite{Benson2003}, a Newton-like solver that incorporates the box constraints on $\rho_h$. Typically, nonlinear convergence is difficult to achieve from a na\"ive initial guess. Hence, the functional $J$ in (\ref{borrvallmin}) is augmented with barrier-like terms. Continuation of the barrier parameter to zero then recovers the solution to the original first-order optimality conditions. A key feature of the deflated barrier method, as required in this work, is the ability to systematically discover multiple solutions of topology optimization problems from the same initial guess. This is achieved via the deflation technique \cite{Farrell2015, Farrell2019b}, \cite[Sec.~3.2]{Papadopoulos2021}. Deflation prevents a Newton-like solver from converging to an already discovered solution by modifying the discretized first-order optimality conditions with a deflation operator. Deflation is extremely cheap to implement (effectively at the same cost as two inner products) and does not affect the conditioning of the linear systems that are solved during the run of the Newton-like solver.  The resulting linear systems arising in the deflated barrier method are solved by a sparse LU factorization with MUMPS \cite{mumps} and PETSc \cite{petsc}. 

There are no known solutions for choices of the inverse permeability, $\alpha$, used in practice. Hence, errors are measured with respect to a heavily-refined finite element solution, which is constructed as follows; first the finite element solutions are computed on a mesh with mesh size $h = 0.028$, using the deflated barrier method. Next, the mesh is adaptively refined three times in areas where the material distribution is between 1/10 and 9/10. Each time the mesh is refined, the coarse-mesh solution is interpolated onto the finer mesh as an initial guess and the first-order optimality conditions are re-solved using the deflated barrier method.

In principle, there can be infinitely many different subsequences of finite element solutions that strongly converge to the same minimizer  of the infinite-dimensional problem   at different convergence rates. Separate subsequences cause difficulties in the interpretation of the convergence plots as they present themselves as oscillations in the error. This is observed in practice and appears to be caused by at least the following two reasons: 
\begin{enumerate}[label=({P}\arabic*)]
\setlength\itemsep{0em}
\item Multiple finite element solutions can exist on the same mesh that represent the same solution  of the infinite-dimensional problem, e.g.~Figure \ref{fig:twostraightchannels}; \label{numproblem1}
\item A fine mesh can align worse than a coarser mesh with the jumps in the material distribution   that solves the infinite-dimensional problem.     \label{numproblem2}
\end{enumerate}
Observation \ref{numproblem1} is not surprising in the context of nonlinear PDEs and nonconvex variational problems. In such cases, an additional selection mechanism is required in order to favor one particular solution over others coexisting on the same mesh. Selection mechanisms are problem-dependent. In the case of nonlinear hyperbolic conservation laws the entropy condition plays this role. In the present context, one might propose choosing the solution, minimizing the modified optimization problem (\ref{BPh}), that attains the smallest objective functional value for $J$, within the basin of attraction of the isolated local minimizer.  For sufficiently small $h$, a minimizer satisfying this selection mechanism must exist. However, it is not necessarily unique and numerically enforcing such a condition can be difficult. In order to promote convergence to the minimizer of (\ref{BPh}) with the smallest value $J$, we interpolate the heavily-refined finite element solutions onto coarser meshes as initial guesses for the deflated barrier method. This strategy was effective in practice. The effects of the second observation \ref{numproblem2} are harder to test. However, in Section \ref{sec:sec:double-pipe}, we attempt to minimize mesh bias by measuring errors on unstructured meshes.\\

\textbf{Code availability:} For reproducibility, an implementation of the deflated barrier method as well as scripts to generate the convergence plots and solutions can be found at \url{https://bitbucket.org/papadopoulos/deflatedbarrier/}. The version of the software used in this paper is archived on Zenodo \cite{dbmcode2021}. 

\subsection{Discontinuous-forcing}
Consider the optimization problem (\ref{borrvallmin}), with a homogeneous Dirichlet boundary condition on $\vect{u}$, $\Omega = (0,1)^2$, volume fraction $\gamma = 1/3$, viscosity $\nu = 1$ and a forcing term  given by
\begin{align}
\vect{f}(x,y) = 
\begin{cases}
(10, 0)^\top & \text{if} \;\; 3/10 < x < 7/10 \;\; \text{and} \;\; 3/10 < y < 7/10,\\
(0, 0)^\top & \text{otherwise}.
\end{cases}
\end{align}
The inverse permeability, $\alpha$, is as given in (\ref{eq:alphachoice}), with $ \overline{\alpha} = 2.5 \times 10^4$ and $q = 1/10$, which satisfies \ref{alpha1}--\ref{alpha5}. Here $q$ is a penalty parameter which controls the level of intermediate values (between zero or one) in the optimal design. Figure \ref{fig:midpointsolns} depicts the material distribution of three minimizers. One local minimizer is in the shape of a figure eight and the two $\mathbb{Z}_2$ symmetric global minimizers are in the shape of annuli. 
\begin{figure}[ht]
\centering
\includegraphics[width = 0.32\textwidth]{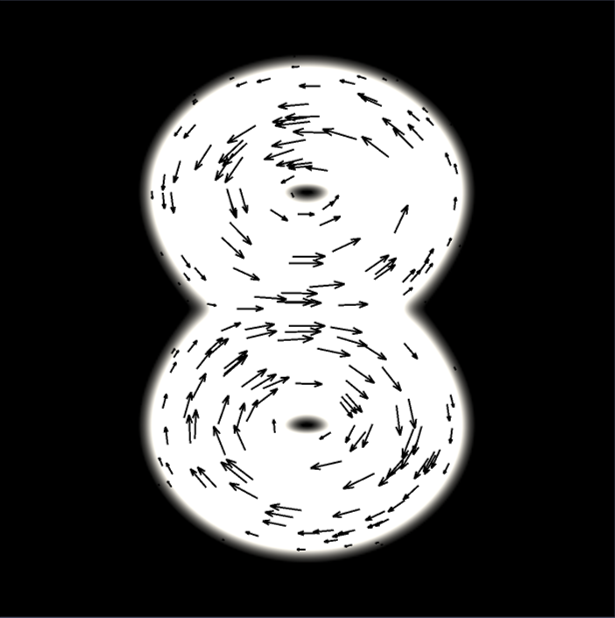}
\includegraphics[width = 0.32\textwidth]{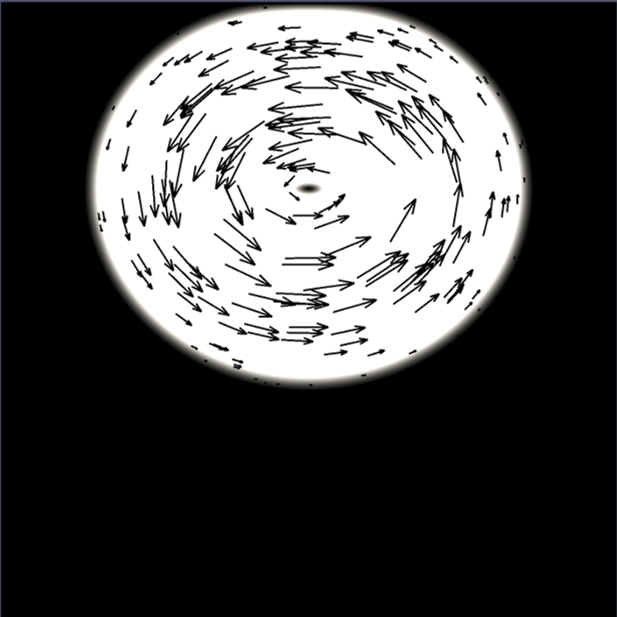}
\includegraphics[width = 0.32\textwidth]{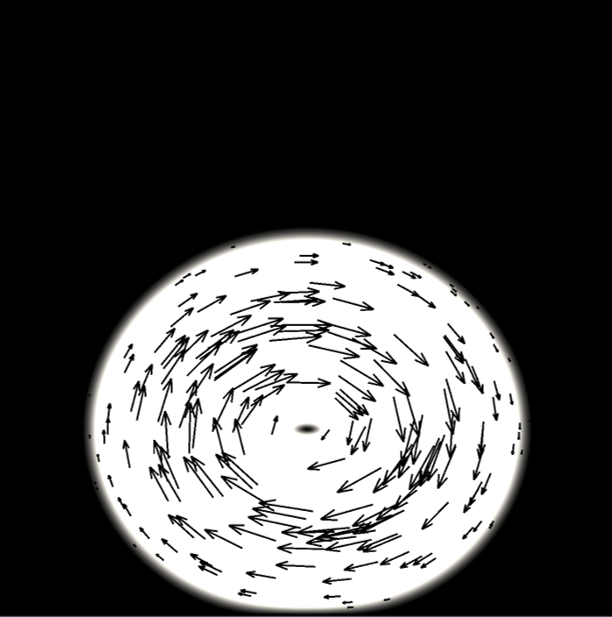}
\caption{The material distribution of a local (left) and the global (middle and right) minimizers of the discontinuous-forcing optimization problem. Black corresponds to a value of $\rho = 0$ and white corresponds to a value of $\rho = 1$, with the gray regions indicating intermediate values. The arrows indicate the velocity profile of the solutions.} 
\label{fig:midpointsolns}
\end{figure}
Since the domain is convex, $\vect{g} = \vect{0}$, and $\vect{f} \in L^2(\Omega)^d$, then, by the regularity results proven in the \ref{sec:ellipticregproof}, $\vect{u} \in H^2(\Omega)^2$ and $p \in H^1(\Omega)$. The conditions of Theorem \ref{th:regularity1} hold and, therefore, $\rho \in H^1(U_\theta)$ for every $\theta > 0$. In this particular example, the support of $\rho$ is compactly contained in the support of the velocity in all three solutions. Therefore, we conclude that $\rho \in H^1(\Omega)$.  

Consider a $(\mathcal{P}_2)^2\times\mathcal{P}_1$ Taylor--Hood finite element discretization for the velocity and the pressure, and a $\mathcal{P}_0$ piecewise constant finite element discretization for the material distribution. Since all three solutions are isolated local minimizers, by Theorem \ref{th:FEMexistence}, there exists a sequence of finite element solutions to the discretized first-order optimality conditions that strongly converges to the figure eight solution, and different sequences of different finite element solutions that strongly converge to the two annulus solutions. Their existence is confirmed in Figure \ref{fig:discconvergenceFenicsDG}. 

Since $\rho \in H^1(\Omega)$ and we are using a $\mathcal{P}_0$ finite element discretization, a na\"ive prediction for the convergence rate of the $L^2$-norm error of the material distribution is $\mathcal{O}(h)$. This rate is observed in the bottom left panel of Figure \ref{fig:discconvergenceFenicsDG}. Moreover, since the minimum regularity of the velocity is $\vect{u} \in H^2(\Omega)^d$, and we are using a $(\mathcal{P}_2)^2$ finite element discretization, a prediction  for the expected convergence rates of the velocity are $\mathcal{O}(h)$ and $\mathcal{O}(h^2)$ for the $H^1$-norm and $L^2$-norm errors of the velocity, respectively. 

 In the standard Stokes system, the regularity of $\vect{u}$ is related to the regularity of the forcing term $\vect{f} \in H^s(\Omega)^d$, such that $\vect{u} \in H^{s+2}(\Omega)^d$ (assuming the domain and boundary data are also suitably regular). Here, the regularity of the forcing term satisfies $s<1/2$. If we assume that the velocity has the additional regularity $\vect{u} \in H^{s+2}(\Omega)^d$, $s \in (0,1/2)$, in this context, a prediction for the upper limit of the convergence rate is $\mathcal{O}(h^r)$ and $\mathcal{O}(h^{t+1})$, for some $r, t \in [1,s+1]$, for the $H^1$-norm and $L^2$-norm errors of the velocity, respectively.  The rates observed in the top panels of Figure \ref{fig:discconvergenceFenicsDG} match this prediction. The $H^1$-norm error is decreasing at a rate slightly faster than $\mathcal{O}(h^{3/2})$ for all three solutions and the $L^2$-norm error convergence rate is $\mathcal{O}(h^2)$ for the figure eight solution and $\mathcal{O}(h^{5/2})$ for the annuli solutions.   We hypothesize that the upper limit of the convergence rate of the $L^2$-norm error of the velocity is bounded by the relatively slower rate of the convergence of the material distribution. 

Finally, since the minimum regularity of the pressure is $p \in H^1(\Omega)$ and the discretization is $\mathcal{P}_1$, a prediction for the convergence rate of the $L^2$-norm error is $\mathcal{O}(h^r)$,  for some $r \in [1,s+1]$.  Initially, the convergence rate is $\mathcal{O}(h^{3/2})$ which matches our na\"ive prediction. However, on finer meshes, the convergence rate increases. We hypothesize that this speedup is artificial and is caused by the lack of resolution of the refined finite element solutions that are being used as proxies for the solutions  of the infinite-dimensional problem  in the error norm estimate. Qualitatively, it can be checked that mesh refinement in areas where the discretized material distribution lies between 1/10 and 9/10 is an ineffective strategy for improving the approximation of the pressure  that solves the infinite-dimensional problem   over the whole domain.
\begin{figure}[ht]
\centering
\includegraphics[width = 0.49\textwidth]{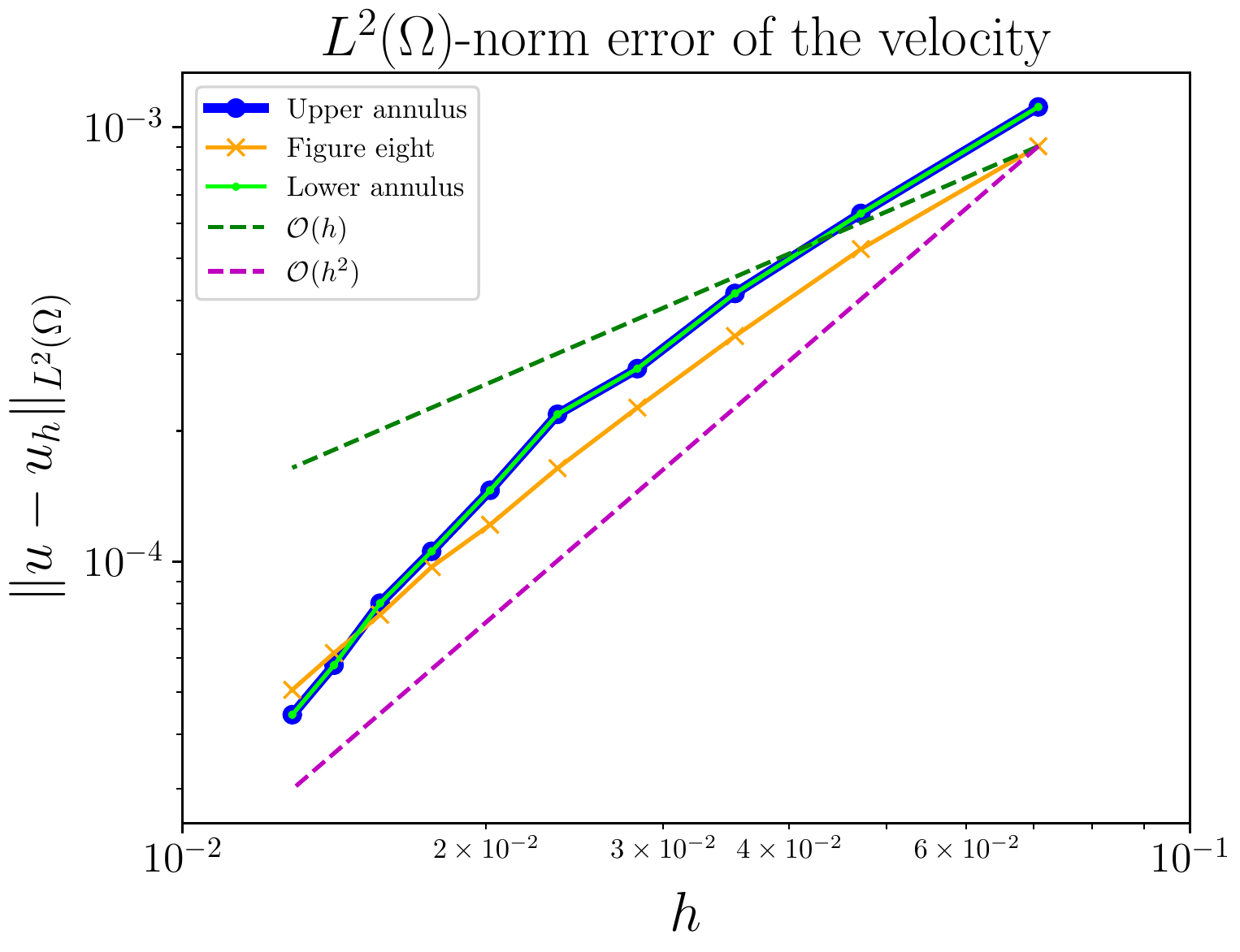}
\includegraphics[width = 0.49\textwidth]{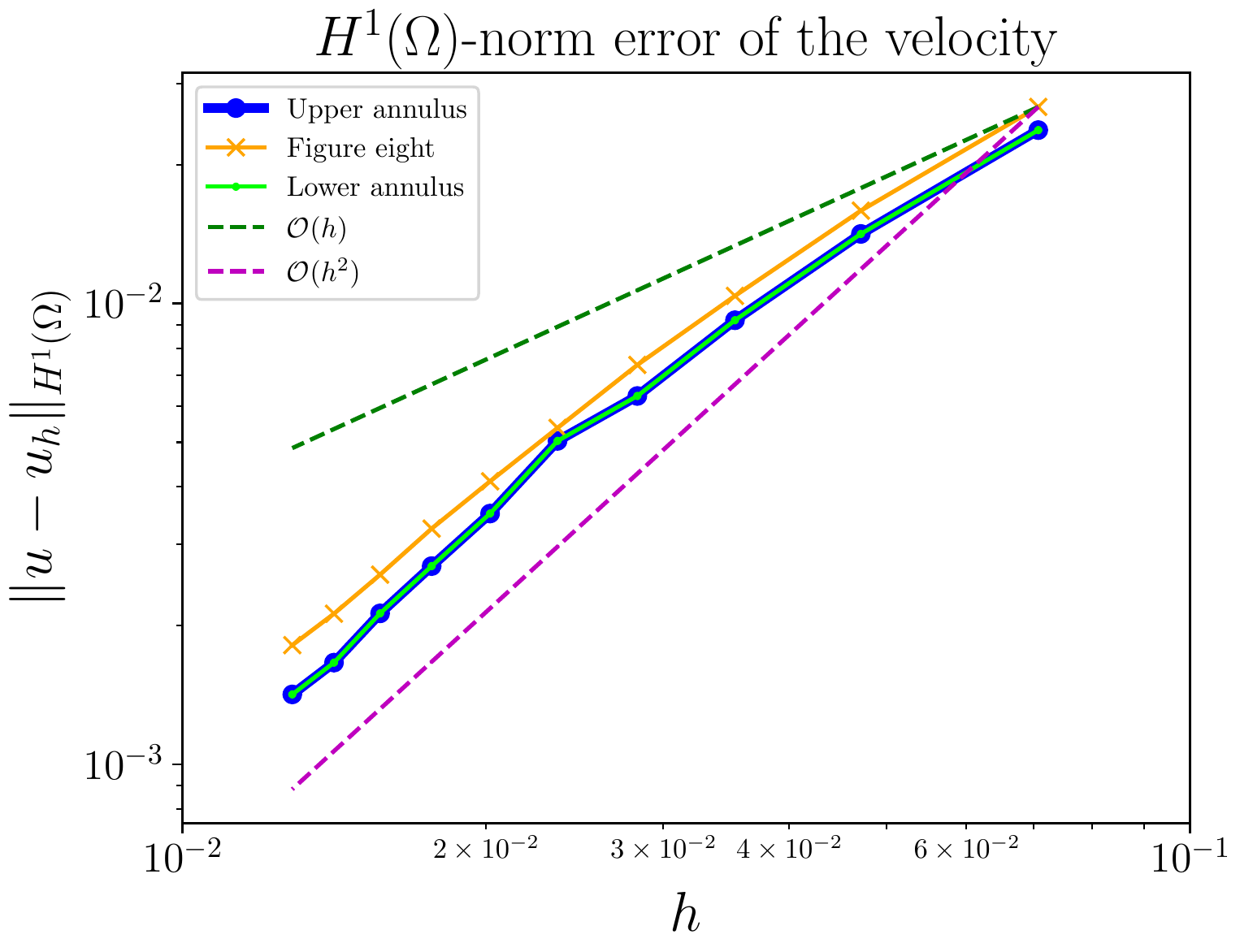}
\includegraphics[width = 0.49\textwidth]{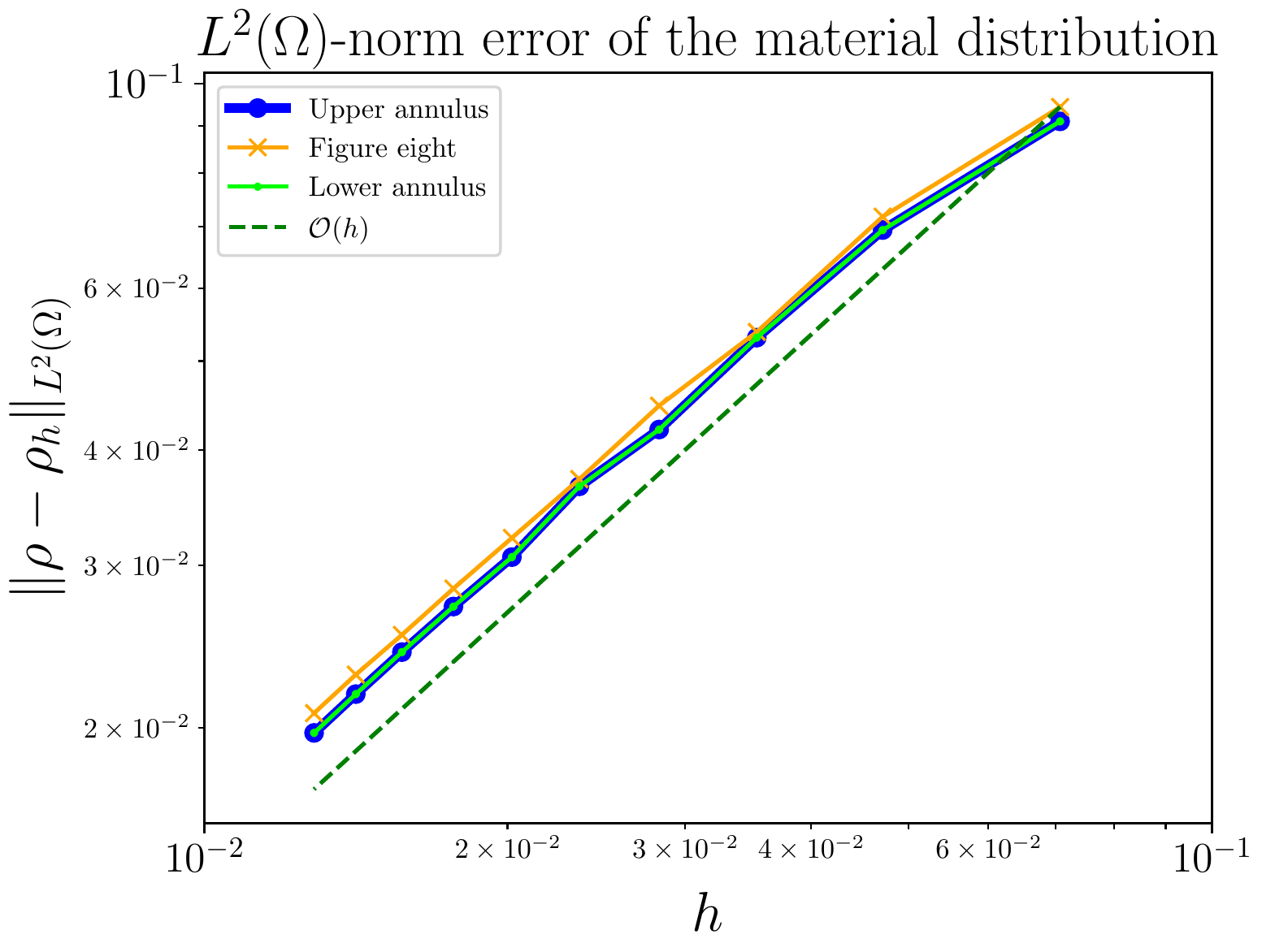}
\includegraphics[width = 0.49\textwidth]{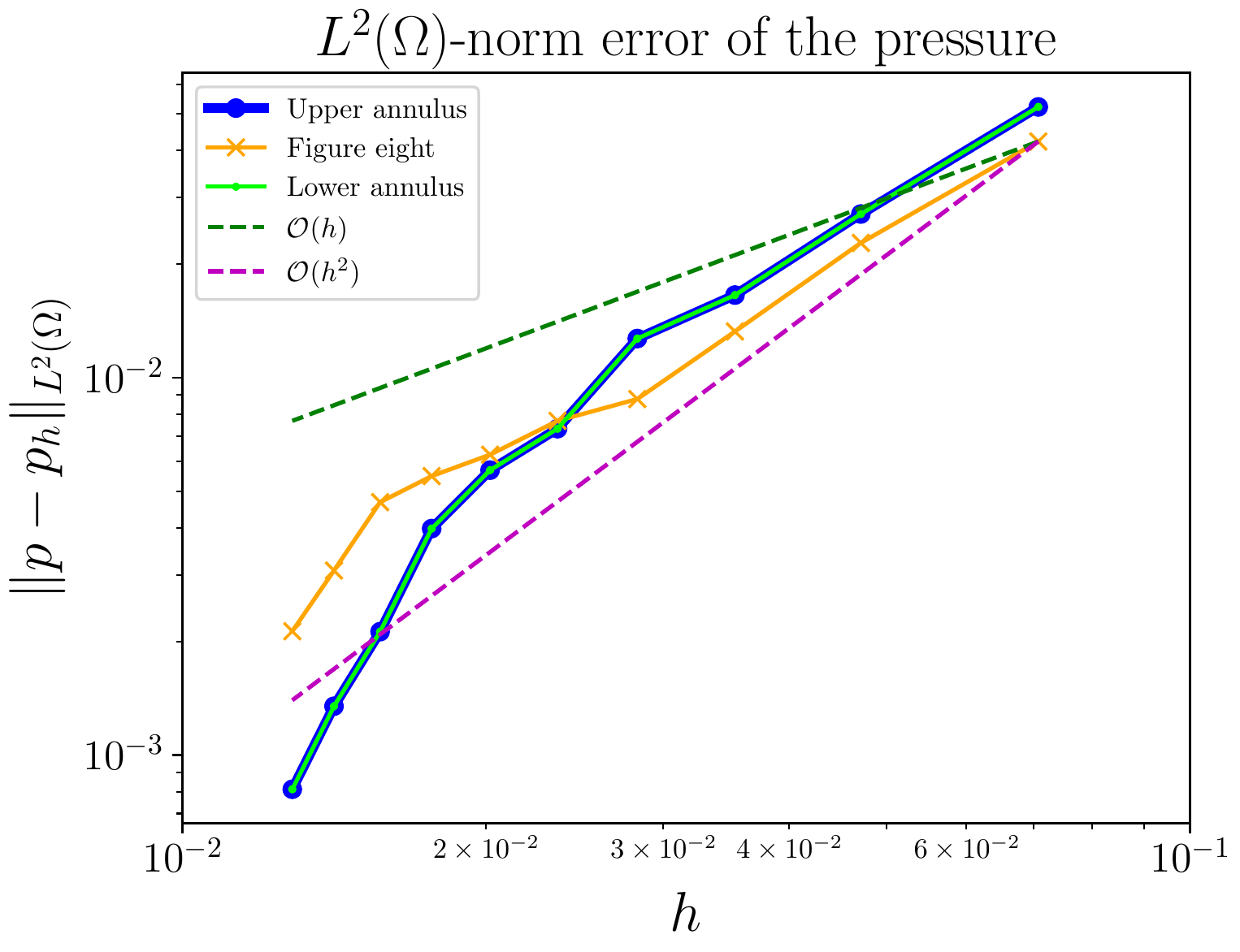}
\caption{The convergence of  $\vect{u}_h$, $\rho_h$, and $p_h$ in the discontinuous-forcing problem for the figure eight and annulus solutions on structured meshes, with a $\mathcal{P}_0 \times (\mathcal{P}_2)^2 \times \mathcal{P}_1$ discretization for $(\rho_h, \vect{u}_h, p_h)$.} 
\label{fig:discconvergenceFenicsDG}
\end{figure}

\subsection{Double-pipe} 
\label{sec:sec:double-pipe}
Consider the optimization problem (\ref{borrvallmin}) \cite[Sec.~4.5]{Borrvall2003}, with two prescribed flow inputs and two prescribed outputs, where $\Omega = (0,3/2) \times (0,1)$, $\gamma = 1/3$, $\vect{f}=(0,0)^\top$ and $\nu = 1$  and the boundary conditions on $\vect{u}$ are given by the boundary data
\begin{align}
\label{eq:doublepipebcs}
 \vect{g}(x,y) = 
\begin{cases}
\left(1-144(y-3/4)^2, 0\right)^\top & \text{if} \;\; 2/3 \leq y \leq 5/6, x = 0 \; \text{or} \; 3/2,\;\; \\
\left(1-144(y-1/4)^2, 0\right)^\top & \text{if} \;\; 1/6 \leq y \leq 1/3, x = 0 \; \text{or} \; 3/2,\;\; \\
(0,0)^\top & \text{elsewhere on} \;\partial \Omega.
\end{cases}
\end{align}
The function $\alpha$ is as given in (\ref{eq:alphachoice}), with $ \overline{\alpha} = 2.5 \times 10^4$ and $q = 1/10$. Figure \ref{fig:doublepipe} visualizes the setup of the problem and depicts the material distribution of the two minimizers. One local minimizer is a straight channel solution and the global minimizer is in the form of a double-ended wrench. 
\begin{figure}[ht]
\centering
\includegraphics[width = 0.4\textwidth]{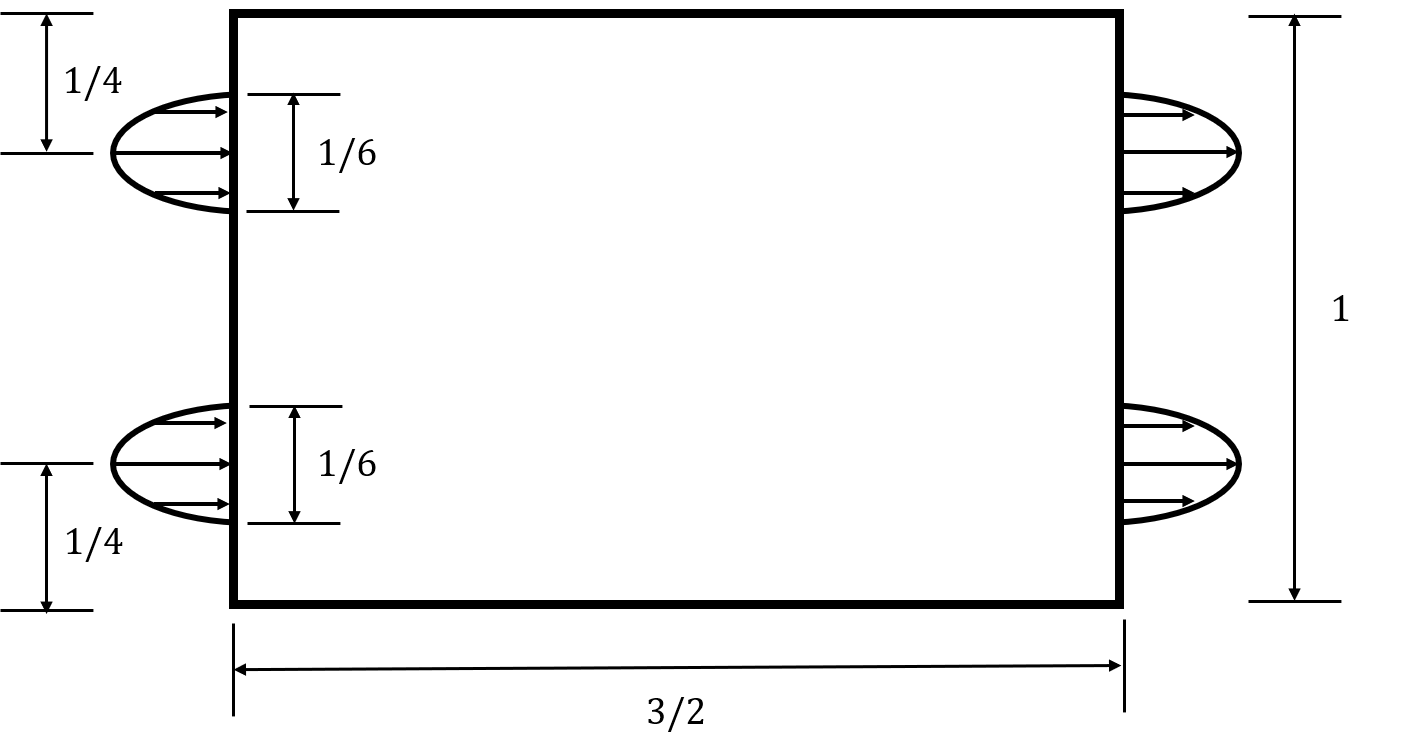}
\includegraphics[width = 0.29\textwidth]{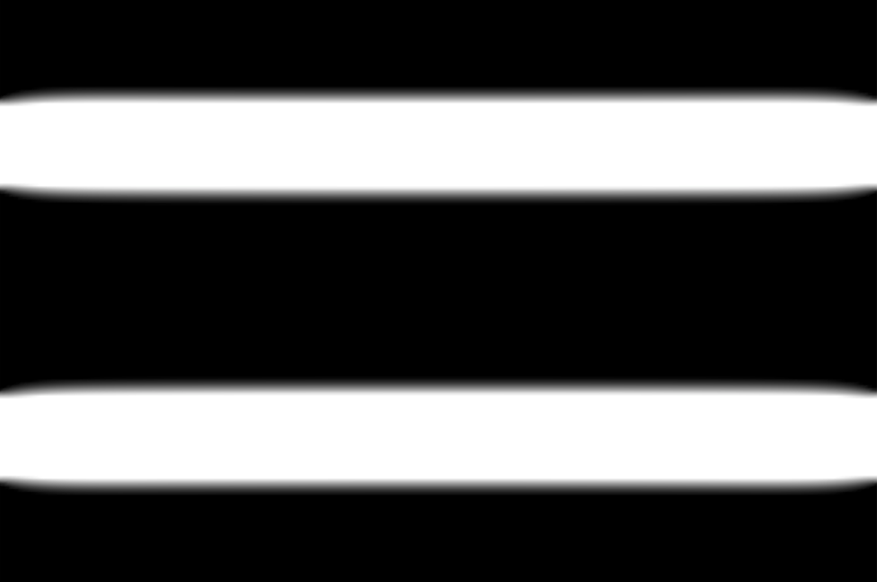}
\includegraphics[width = 0.29\textwidth]{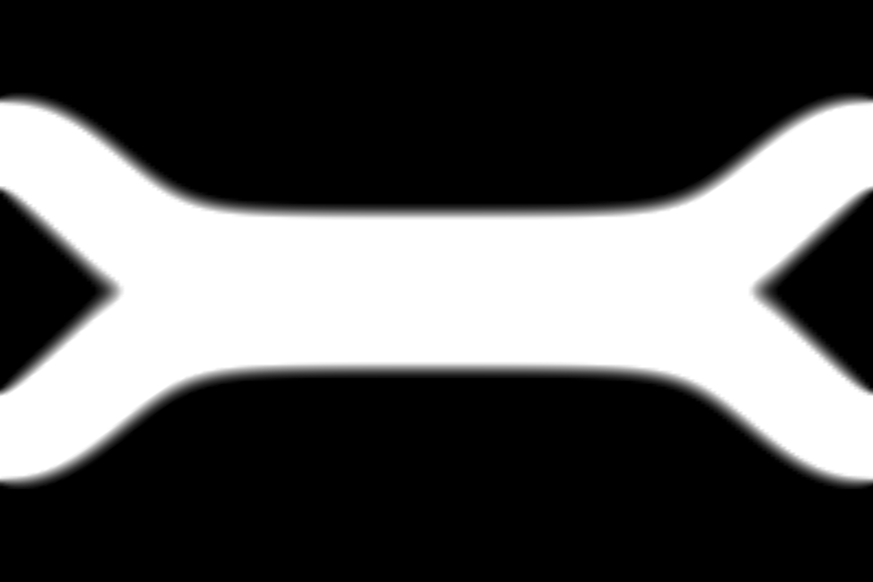}
\caption{The setup of the double-pipe problem is shown on the left. The middle and right figures depict the material distribution of the straight channel and double-ended wrench solutions, respectively. Black corresponds to a value of $\rho = 0$ and white corresponds to a value of $\rho = 1$, with the gray regions indicating intermediate values.} 
\label{fig:doublepipe}
\end{figure}

We employ a $(\mathcal{P}_2)^2 \times \mathcal{P}_1$ Taylor--Hood finite element discretization for the velocity and the pressure, and a $\mathcal{P}_1$ continuous piecewise linear finite element discretization for the material distribution. This example satisfies all the conditions of Theorem \ref{th:FEMexistence} and, for both minimizers, we numerically verify that there exists a sequence of finite element solutions that strongly converges to it in Figure \ref{fig:doublepipeconvergenceFenicsUnstructured}. 
\begin{figure}[ht!]
\centering
\includegraphics[width = 0.49\textwidth]{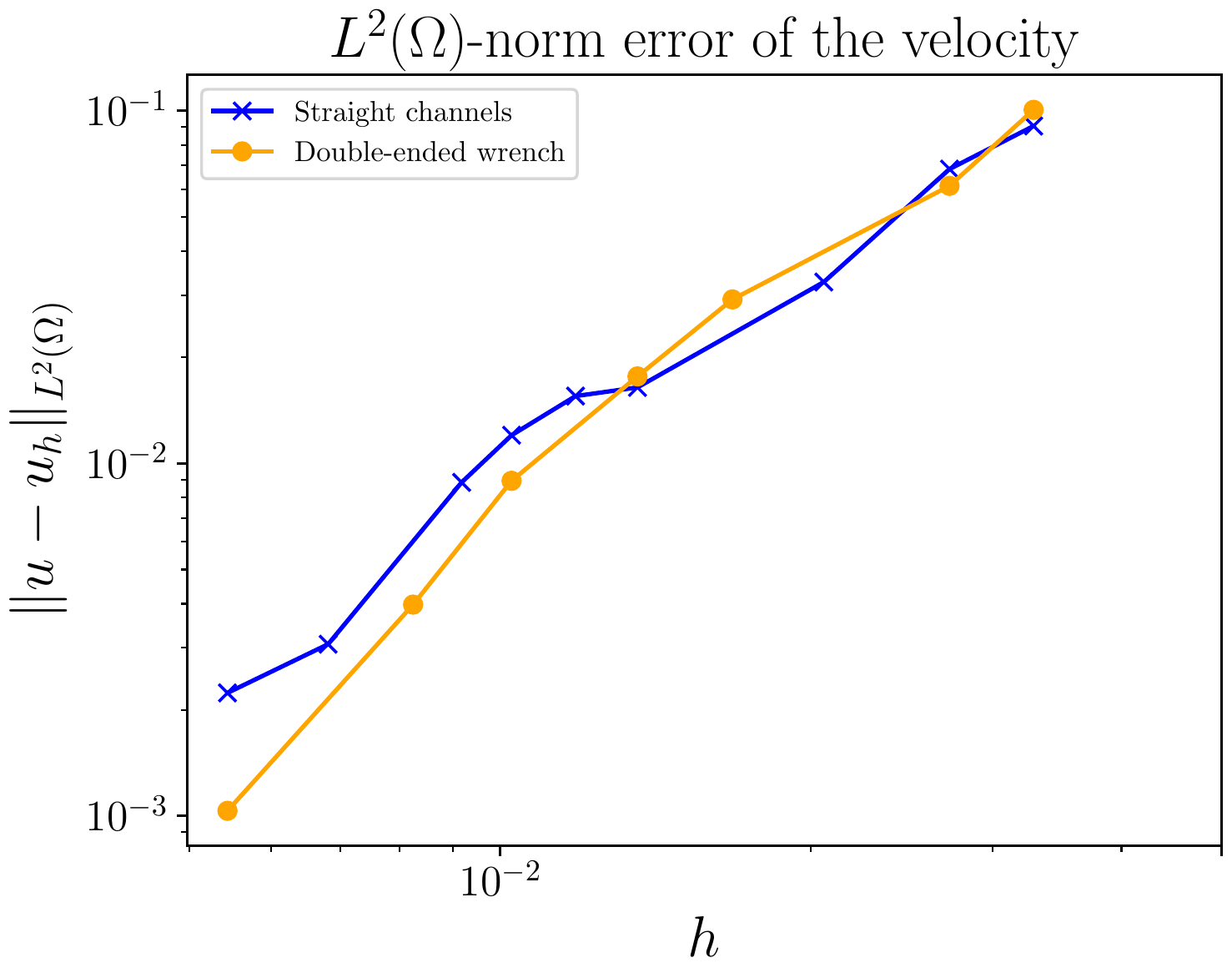}
\includegraphics[width = 0.49\textwidth]{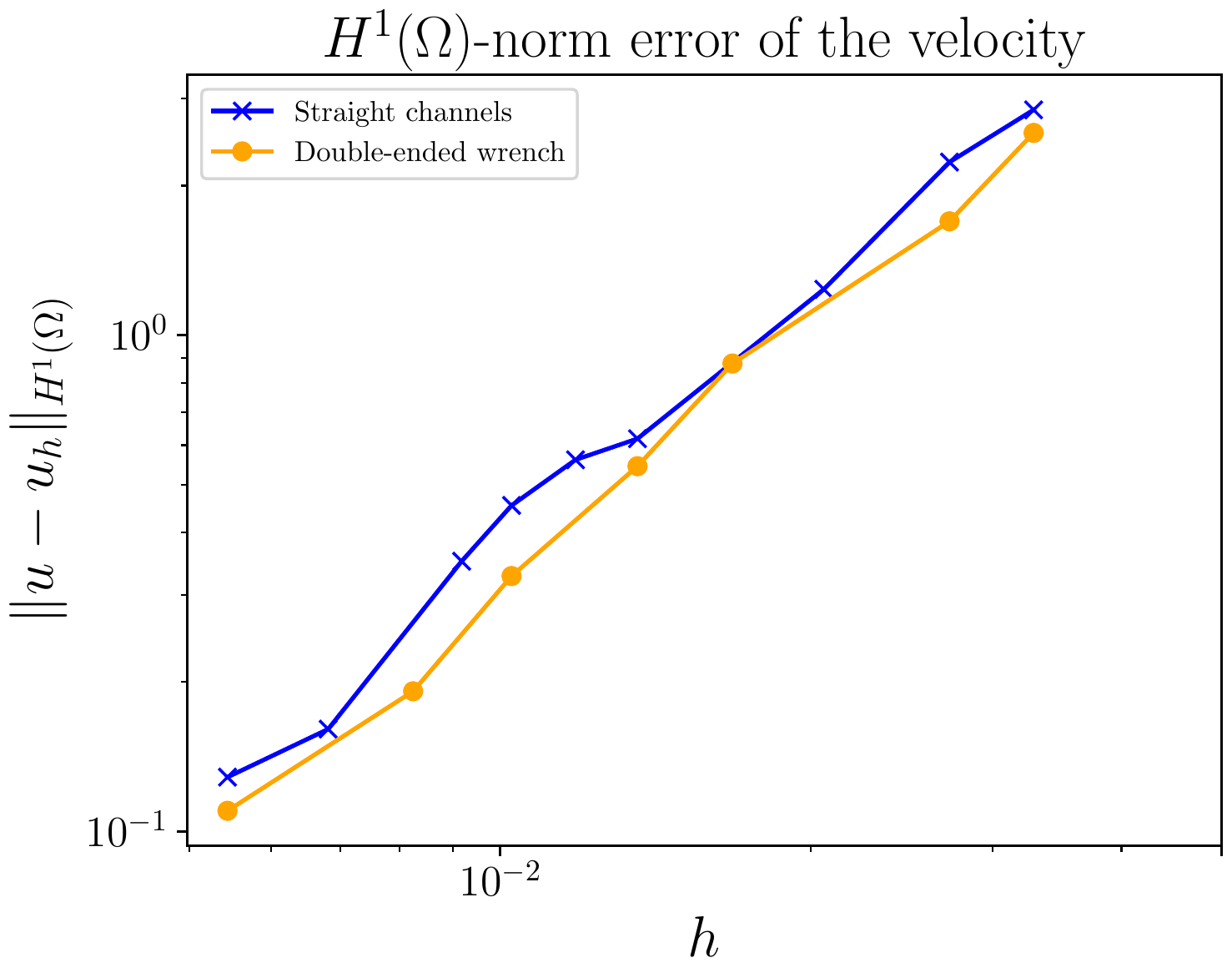}
\includegraphics[width = 0.49\textwidth]{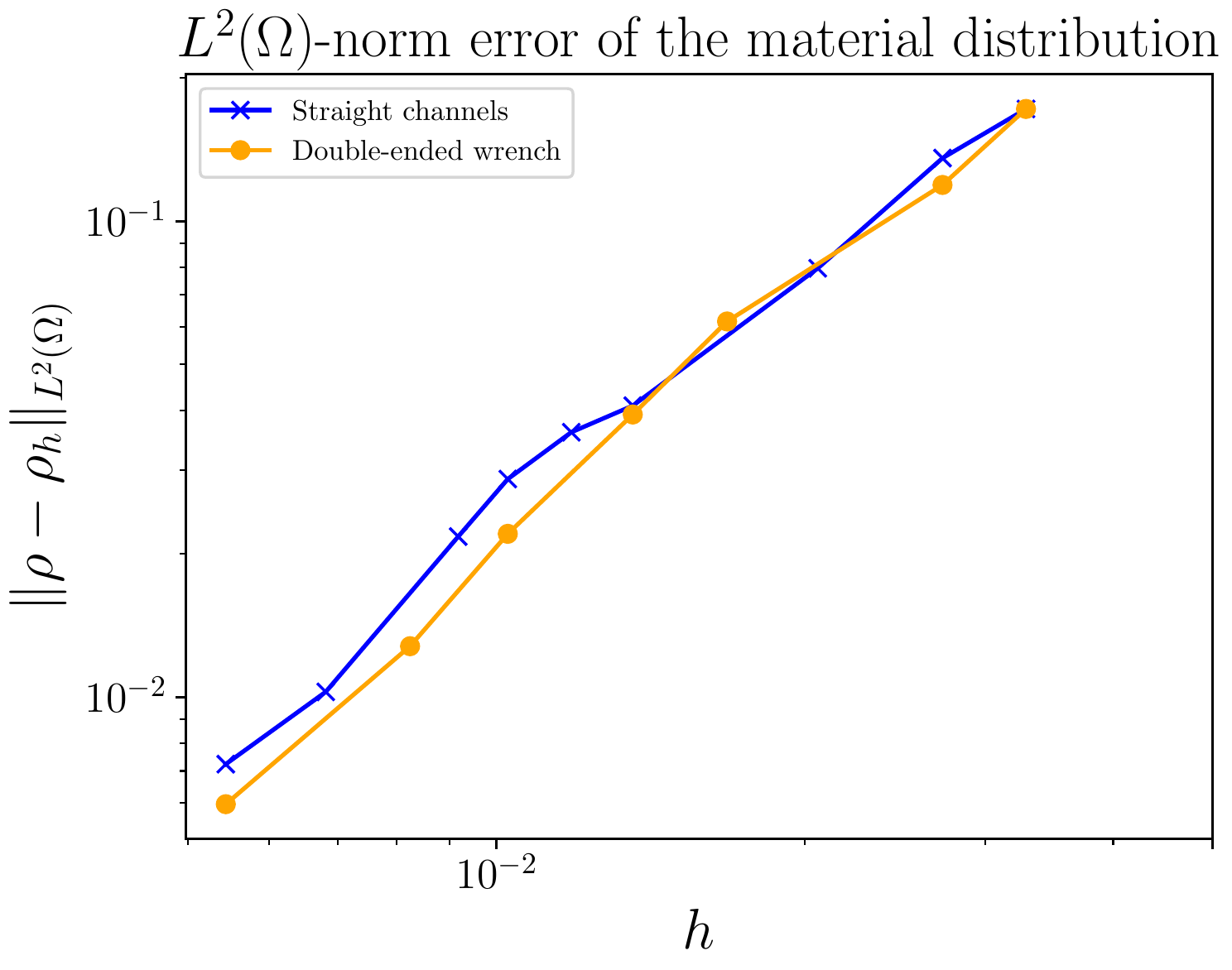}
\includegraphics[width = 0.49\textwidth]{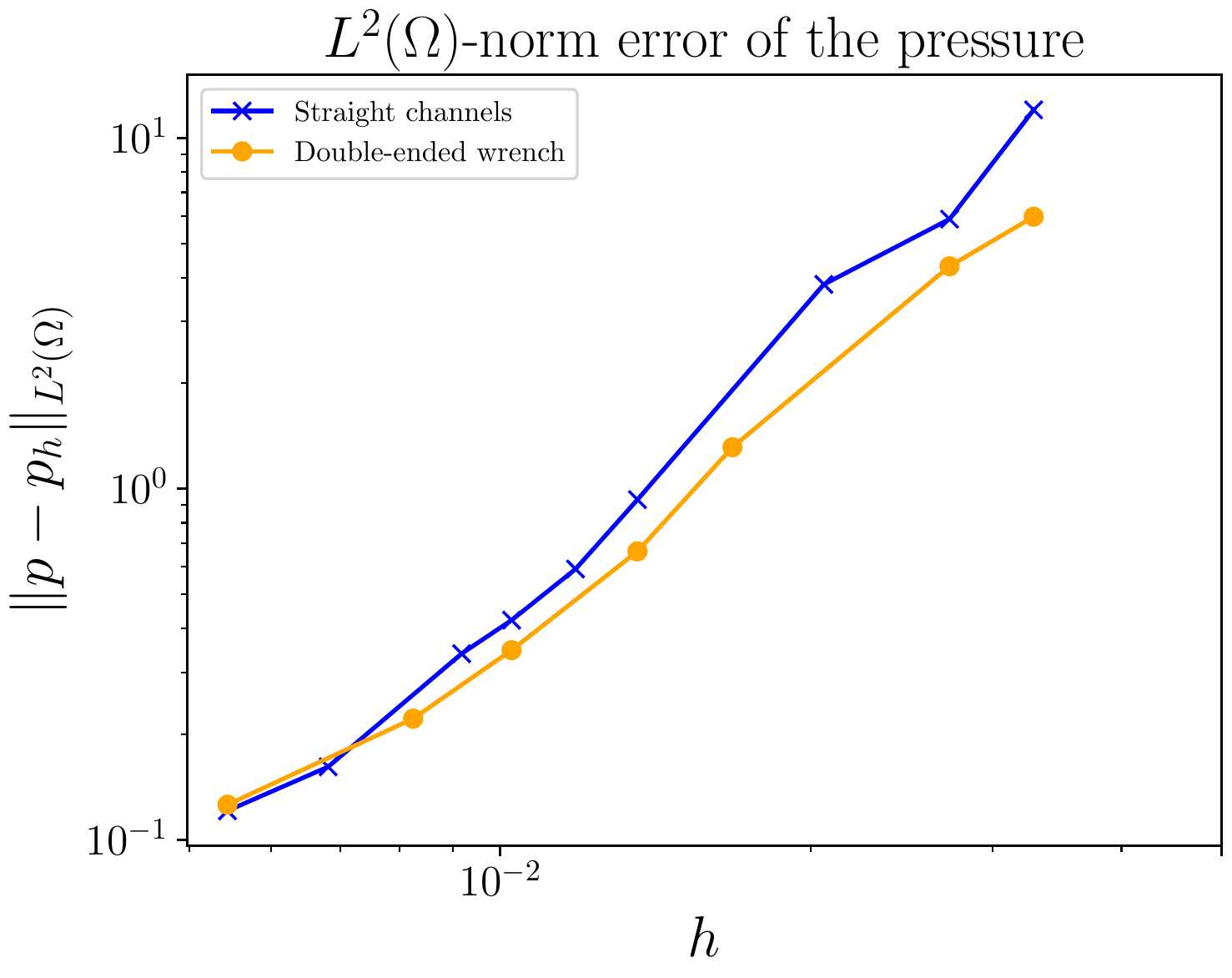}
\caption{The convergence of $\vect{u}_h$, $\rho_h$, and $p_h$ for the double-pipe problem for both the straight channel and double-ended wrench solutions on an unstructured mesh with a $\mathcal{P}_1 \times (\mathcal{P}_2)^2 \times \mathcal{P}_1$ discretization for $(\rho_h, \vect{u}_h, p_h)$.} 
\label{fig:doublepipeconvergenceFenicsUnstructured}
\end{figure}

As mentioned earlier in this section, it may be possible to find a sequence of mesh sizes, $(h_i)$, such that there exist two different sequences of finite element solutions that strongly converge to the same isolated minimizer. In Figure \ref{fig:twostraightchannels}, we depict two different straight channel finite element solutions that exist on the same unstructured mesh where $h=0.04$. Both solutions satisfy the discretized first-order optimality conditions (\ref{FOC1h})--(\ref{FOC3h}) and both locally minimize $J(\vect{v}_h, \eta_h)$. Choosing one over the other would change the convergence pattern of the strongly converging sequence. This may cause difficulty in practice, as optimization strategies are unlikely to discover the discretized global minimum without additional selection mechanisms. 
\begin{figure}[ht!]
\centering
\includegraphics[height = 0.3\textwidth]{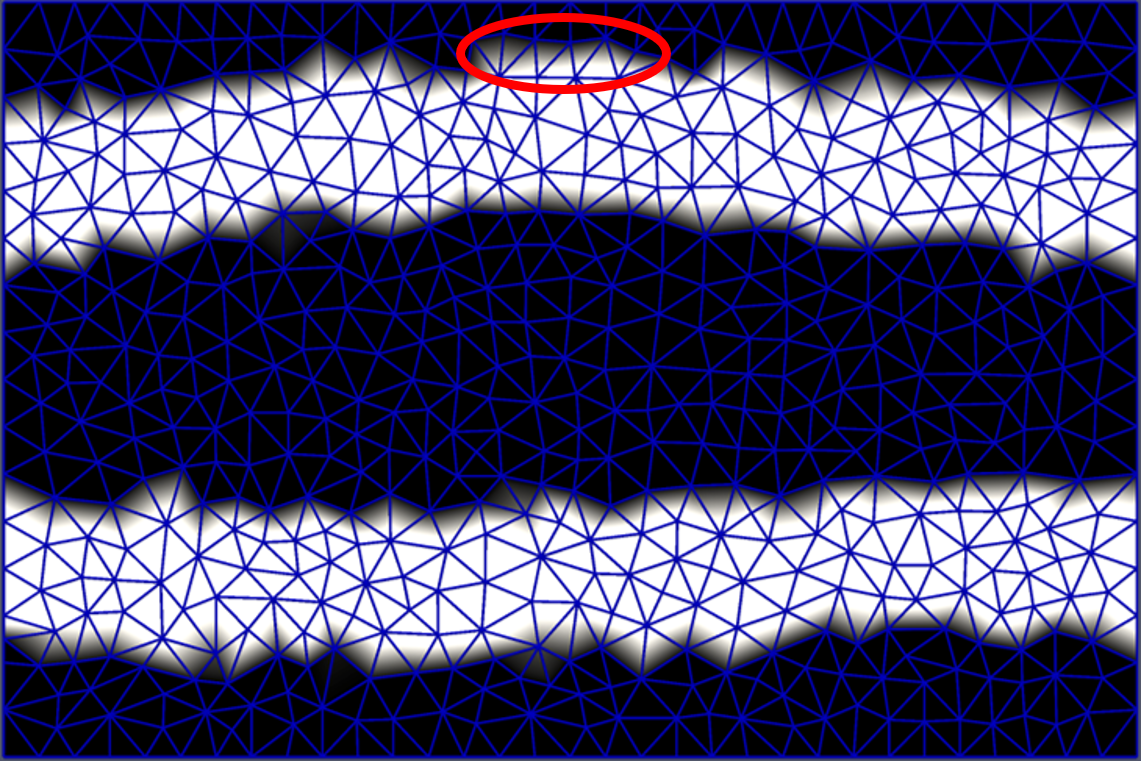}
\includegraphics[height = 0.3\textwidth]{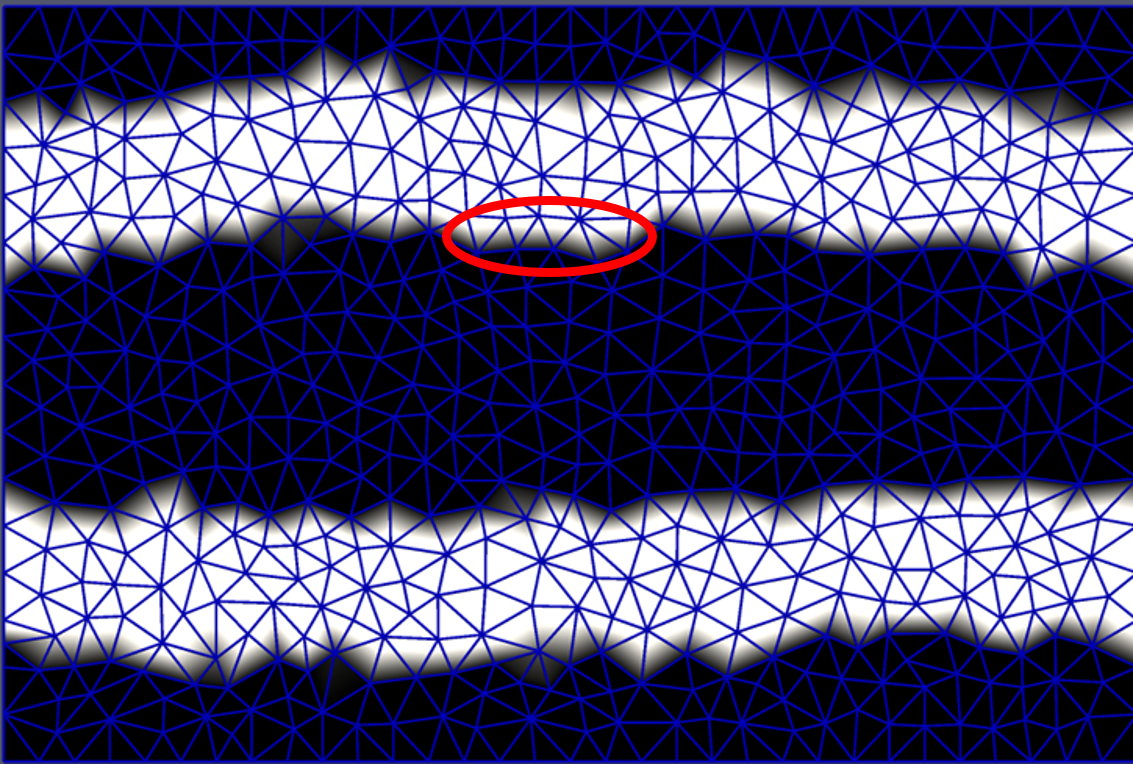}
\caption{Two different straight channel finite element solutions of the double-pipe optimization problem that exist on the same unstructured mesh where $h=0.04$. The differences can be spotted at the midway point of the top channel.} 
\label{fig:twostraightchannels}
\end{figure}

\section{Conclusions}
In this work we have studied the fluid topology optimization model of Borrvall and Petersson \cite{Borrvall2003}. In the case of a homogeneous Dirichlet boundary condition and under a mild convexity assumption on the inverse permeability term, $\alpha$, we have shown that the material distribution necessarily lives in the Sobolev space $H^1$ inside any compact subset of the support of the velocity. Moreover, we have formally treated the nonconvexity of the optimization problem (including the case of inhomogeneous Dirichlet boundary conditions) and have shown that, given an isolated minimizer  of the infinite-dimensional problem,  there exists a sequence of discretized solutions, satisfying the associated first-order optimality conditions, that \emph{strongly} converges to the minimizer in the appropriate spaces.~We have numerically verified that these sequences exist and have discussed the observed convergence rates. 

\begin{appendix}
\gdef\thesection{\appendixname\!\!}
\section[Elliptic regularity]{Elliptic regularity of the generalized Stokes system with Dirichlet boundary data}
\label{sec:ellipticregproof}
\begin{lemma}
Let the domain $\Omega$ be either a convex polygon in two dimensions or a convex polyhedron in three dimensions and consider the triple $(\vect{u}, \rho, p) \in H^1_{\vect{g}}(\Omega)^d \times C_\gamma \times L^2_0(\Omega)$ that satisfies (\ref{FOC1})--(\ref{FOC3}). Suppose that the forcing term $\vect{f} \in L^2(\Omega)^d$ and the boundary datum $\vect{g}$ is the boundary trace of a function $\hat{\vect{g}} \in H^2(\Omega)^d$ on the boundary $\partial \Omega$  and satisfies $\int_{\partial \Omega} \vect{g} \cdot \vect{n} \, \mathrm{d}s = 0$.  Then $\vect{u} \in H^2(\Omega)^d$ and $p \in H^1(\Omega)$. 
\end{lemma}

\begin{proof}
The idea of the proof is to reduce the system (\ref{FOC1})--(\ref{FOC2}) to the standard Stokes system with a homogeneous Dirichlet boundary condition and invoke the regularity results of Kellogg and Osborn \cite{Kellogg1976} and in three dimensions the results found in Kozlov et al.~\cite{Kozlov1994} and  Maz'ya and Shaposhnikova \cite{Mazya2011}.  

Let $\vect{w} := \vect{u} - \vect{\hat{g}}$. Since the trace operator is a linear operator, we see that $\vect{w}|_{\partial \Omega} = (\vect{u} - \vect{\hat{g}})|_{\partial \Omega} = \vect{g} - \vect{g}= \vect{0}$. Since $\vect{u} \in H^1(\Omega)^d$ and $\vect{\hat{g}} \in H^2(\Omega)^d$, then $\vect{w} \in H^1_0(\Omega)^d$. 

By substituting $\vect{w}$ into  (\ref{FOC1})--(\ref{FOC2}), we see that (\ref{FOC1})--(\ref{FOC2}) is equivalent to finding $(\vect{w}, p) \in H^1_{0}(\Omega)^d  \times L^2_0(\Omega)$ that satisfies for all $(\vect{v}, q) \in H^1_{0}(\Omega)^d  \times L^2_0(\Omega)$:
\begin{align}
\int_\Omega  \nabla \vect{w} : \nabla \vect{v} - p \; \mathrm{div}(\vect{v}) \; \mathrm{d}x &= \int_\Omega (\vect{f} - \alpha(\rho) (\vect{w} + \vect{\hat{g}}) ) \cdot \vect{v} - \nabla \vect{\hat{g}} : \nabla \vect{v} \; \mathrm{d}x, \label{gstokes1}\\
\int_\Omega  q \; \mathrm{div}(\vect{w} + \vect{\hat{g}}) \; \mathrm{d}x &= 0. \label{gstokes2}
\end{align}
Define $\vect{\hat{f}}$ as $\vect{\hat{f}}:=\vect{f} - \alpha(\rho) (\vect{w} + \vect{\hat{g}}) + \Delta \vect{\hat{g}}$. Since $\vect{f} \in L^2(\Omega)^d$, $\alpha(\rho) \in L^\infty(\Omega)$, $\vect{w} \in H^1_0(\Omega)^d$, and $\vect{\hat{g}} \in H^2(\Omega)^d$, then $\vect{\hat{f}} \in L^2(\Omega)^d$. By an application of integration by parts on the final term on the right-hand side of (\ref{gstokes1}), we see that (\ref{gstokes1})--(\ref{gstokes2}) is equivalent to finding $(\vect{w}, p) \in H^1_{0}(\Omega)^d  \times L^2_0(\Omega)$ that satisfies for all $(\vect{v}, q) \in H^1_{0}(\Omega)^d  \times L^2_0(\Omega)$:
\begin{align}
\int_\Omega  \nabla \vect{w} : \nabla \vect{v} - p \; \mathrm{div}(\vect{v}) \; \mathrm{d}x &= \int_\Omega \vect{\hat{f}} \cdot \vect{v} \; \mathrm{d}x, \label{gstokes3}\\
\int_\Omega  q \; \mathrm{div}(\vect{w}) \; \mathrm{d}x &=  \int_\Omega q \phi \, \mathrm{d}x,  \label{gstokes4}
\end{align}

where $\phi = -\mathrm{div}(\vect{\hat{g}})$ a.e.~and the divergence theorem implies that
\begin{align}
\int_\Omega \phi \, \mathrm{d}x = - \int_{\partial \Omega} \vect{g} \cdot \vect{n} \, \mathrm{d}s = 0.
\end{align}

We note that (\ref{gstokes3})--(\ref{gstokes4}) is the standard Stokes system with a homogeneous Dirichlet boundary condition and forcing term $\vect{\hat{f}} \in L^2(\Omega)^d$. Therefore, by the elliptic regularity of the Stokes system \cite{Kellogg1976, Kozlov1994} and \cite[Th.~13]{Mazya2011}, $\vect{w} \in H^2(\Omega)^d$ and $p \in H^1(\Omega)$. Since $\vect{u} = \vect{w} + \vect{\hat{g}}$ and $\vect{\hat{g}} \in H^2(\Omega)^d$, we conclude that $\vect{u} \in H^2(\Omega)^d$. 
\end{proof}
\end{appendix}

\bibliography{references}

\end{document}